\documentclass[10pt]{amsart}

\usepackage{amssymb}
\usepackage[leqno]{amsmath}
\usepackage{mathrsfs}
\usepackage{stmaryrd}
\usepackage{chemarrow}
\usepackage[norelsize]{algorithm2e}
\usepackage{enumerate}
\usepackage{graphicx}
\usepackage[pdftex,bookmarksnumbered,bookmarksopen,colorlinks,linkcolor=red,anchorcolor=black,citecolor=blue,urlcolor=blue]{hyperref}
\usepackage[all]{xy}
\usepackage{tikz}
\usetikzlibrary{arrows}

\usepackage{multirow}

\newcommand{\XH}[1]{\textcolor{black}{#1}}

  \newcounter{mnote}
  \let\oldmarginpar\marginpar
    \renewcommand\marginpar[1]{\-\oldmarginpar[\raggedleft\footnotesize #1]%
    {\raggedright\footnotesize #1}}

\newtheorem{theorem}{Theorem}[section]
\newtheorem{lemma}[theorem]{Lemma}

\newtheorem{remark}[theorem]{Remark}

\newcommand{\bs}{\boldsymbol}

\renewcommand{\div}{\operatorname{div}}

\numberwithin{equation}{section}

\begin{document}
\title[$H^m$-nonconforming VEM with $m>n$]{Nonconforming Virtual Element Method for $2m$-th Order Partial Differential Equations in $\mathbb R^n$ with $m>n$}
%\title[Nonconforming VEM for biharmonic equation]{$H^m$-nonconforming Virtual Element for $2m$-th Order Partial Differential Equations in $\mathbb R^n$}
%\author{Shuhao Cao}%
%\address{Department of Mathematics, University of California at Irvine, Irvine, CA 92697, USA}%
%\email{shuhao.cao@uci.edu}%
%\author{Long Chen}%
%\address{Department of Mathematics, University of California at Irvine, Irvine, CA 92697, USA}%
%\email{chenlong@math.uci.edu}%
\author{Xuehai Huang}%
\address{School of Mathematics, Shanghai University of Finance and Economics, Shanghai 200433, China}%
\email{huang.xuehai@sufe.edu.cn}%

%\thanks{The first author was supported by.}
%\thanks{
%The first author was supported by the National Science Foundation (NSF) DMS-1418934, and
%in part by the Sea Poly Project of Beijing Overseas Talents and the National Natural Science Foundation of China Project 11671159.}
\thanks{The author was supported by the National Natural Science Foundation of China Project 11771338, and the Fundamental Research Funds for the Central
Universities 2019110066.}

\subjclass[2010]{
%65N55;   %% Multigrid methods; domain decomposition
%65F10;   %%  Iterative methods for linear systems
65N30;   %%  Finite elements, Rayleigh-Ritz and Galerkin methods, finite methods;
65N12;   %%  Stability and convergence of numerical methods;
65N22;   %%  Solution of discretized equations
%65N15;   %%  Error bounds
}

\begin{abstract}
The $H^m$-nonconforming virtual elements of any order $k$ on any shape of polytope in $\mathbb R^n$ with constraints $m> n$ and $k\geq m$ are constructed in a universal way.
A generalized Green's identity for $H^m$ inner product with $m>n$ is derived, which is essential to devise the $H^m$-nonconforming virtual elements.
By means of the local $H^m$ projection and a stabilization term using only the boundary degrees of freedom, the $H^m$-nonconforming virtual element methods are proposed to approximate solutions of the $m$-harmonic equation.
The norm equivalence of the stabilization on the kernel of the local $H^m$ projection is proved by using the bubble function technique, the Poincar\'e inquality and the trace inequality, which implies the well-posedness of the virtual element methods.
The optimal error estimates for the $H^m$-nonconforming virtual element methods are achieved from an estimate of the weak continuity and the
error estimate of the canonical interpolation. Finally, the implementation of the nonconforming virtual element method is discussed.
\end{abstract}
\maketitle

%\tableofcontents

%\input{introduction}
\section{Introduction}

The $H^m$-nonconforming virtual elements of order $k\in\mathbb N$ on a very general polytope $K\subset\mathbb R^n$ in any dimension and any order with constraints $m\leq n$ and $k\geq m$ have been devised in \cite{ChenHuang2019}.
\XH{While an important case $m=3$ and $n=2$, i.e. the triharmonic equation in two dimensions is not involved in \cite{ChenHuang2019}.}
To this end, and also
for theoretical considerations, we will study the $H^m$-nonconforming virtual element $(K, \mathcal N_K, V_K)$ for case $m>n$ in this paper, which can be considered as the second part of the work \cite{ChenHuang2019}.
Here $\mathcal N_K$ is the set of degrees of freedom, and $V_K$ the finite-dimensional space of shape functions.
The virtual element can be defined on polytopes of any shape, and thus allows the division of the domain into different type of polytopes \XH{\cite{BeiraodaVeigaBrezziCangianiManziniEtAl2013,BeiraodaVeigaBrezziMariniRusso2014}}, which makes the discrete method easier to
capture the singularity of the solution.
%One feature of the virtual element is that the space of shape functions and its bases are not explicitly known,
%but we can compute the projection of the shape functions onto some space of polynomials by using the degrees of freedom $\mathcal N_K$.
The key feature of the virtual element method is that it is completely determined by the degrees of freedom, and the virtual element space is only used for the analysis rather than entering the discrete method for elliptic problems.

%a universal way
It is arduous to design $H^m$-conforming or nonconforming finite elements for large $k$ and $m$, especially $m>n$.
With the help of the bubble functions,
Wang and Xu constructed the minimal $H^m$-nonconforming finite elements on simplices in any dimension with $m=n+1$ in~\cite{WuXu2019}.
Enriching the bilinear form with few interior penalty terms, they proposed a family of interior penalty nonconforming finite element
methods for arbitrary $m$ and $n$ in~\cite{WuXu2017}, whose shape functions are polynomials of the smallest degree $m$.
In two and three dimensions,  $H^m$-conforming finite elements for arbitrary $m$ with shape functions being lower order polynomials were devised in \cite{BrambleZlamal1970,Zenisek1970} for triangular meshes, in \cite{Zenisek1974a} for tetrahedral meshes and in \cite{HuZhang2015a} for rectangular meshes. And some $H^m$-nonconforming elements of lower degree on triangular meshes  for any $m$ were studied in~\cite{HuZhang2017}.
In addition to standard conforming and nonconforming finite element methods, a $C^0$ interior penalty method in \cite{GudiNeilan2011} and a cubic $H^3$-nonconforming macro-element method in \cite{HuZhang2019} were developed  for a sixth-order elliptic equation in two dimensions, and some mixed finite element methods were advanced in \cite{DroniouIlyasLamichhaneWheeler2019,Gallistl2017,Schedensack2016} for $2m$th-order elliptic equations with $m>n$.
When $m\leq n$,
We refer to \cite{ArgyrisFriedScharpf1968,Zhang2009a,Zhang2010,HuHuangZhang2011} for $H^m$-conforming finite elements, and
\cite{WangXu2013,WangXu2006} for $H^m$-nonconforming finite elements.

Constructing the $H^m$-conforming or nonconforming finite elements on general polytopes for arbitrary $k$, $m$ and $n$ in a universal way is extremely difficult, while it is possible for the virtual elements. As a matter of fact, $H^m$-conforming virtual elements in two dimensions with arbitrary $m$ were designed in \XH{\cite{BeiraoManzini2014,AntoniettiManziniVerani2019,BrezziMarini2013}}, which were nontrivial to extend to higher dimension $n>2$. \XH{ We refer to \cite{BeiraodaVeigaDassiRusso2019} for the $H^2$-conforming virtual elements in three dimensions.} While in \cite{AyusodeDiosLipnikovManzini2016,ChenHuang2019,ZhaoChenZhang2016,AntoniettiManziniVerani2018,ZhaoZhangChenMao2018},
$H^m$-nonconforming virtual elements on general polytopes in any dimension $n$ with constraint $m\leq n$ were studied in details.
By the way, we refer to~\cite{Wang2019} for an $H^1$-nonconforming Crouzeix-Raviart type element on polygonal meshes.

In order to construct the $H^m$-nonconforming virtual element in any order on the polytope with any shape in any dimension for $m>n$, by adopting the integration by parts, we first prove the following generalized Green's identity for the $H^m$ space
\begin{align}
(\nabla^mu, \nabla^mv)_K =& \left((-\Delta)^mu,v\right)_K+ \sum_{j=1}^{n-1}\sum_{F\in\mathcal F^j(K)}\sum_{\alpha\in A_{j}\atop|\alpha|\leq m-j}\Big ( D^{2m-j-|\alpha|}_{F, \alpha}(u), \;\frac{\partial^{|\alpha|}v}{\partial\nu_{F}^{\alpha}}\Big )_F  \notag\\
 &+ \sum_{\delta\in\mathcal F^{n}(K)}\sum_{\alpha\in A_{n}\atop|\alpha|\leq m-n}D^{2m-n-|\alpha|}_{\delta,\alpha}(u)\frac{\partial^{|\alpha|}v}{\partial\nu_{\delta}^{\alpha}}(\delta),\label{eq:HmGreenIntro}
\end{align}
where $\mathcal F^j(K)$ is the set of all $(n-j)$-dimensional faces
of the polytope $K$, $A_{j}$ the set consisting of all $n$-dimensional multi-indexes $\alpha=(\alpha_1, \cdots, \alpha_n)$ with $\alpha_{j+1}=\cdots=\alpha_n=0$,  $D^{2m-j-|\alpha|}_{F, \alpha}(u)$ some $(2m-j-|\alpha|)$-th order derivative of $u$ on $F$, and $\frac{\partial^{|\alpha|}v}{\partial\nu_{F}^{\alpha}}$ the multi-indexed normal derivative on $F$. \XH{Compared to the generalized Green's identity for $m\leq n$ in \cite{ChenHuang2019}, the generalized Green's identity \eqref{eq:HmGreenIntro} involves the additional term $D^{2m-n-|\alpha|}_{\delta,\alpha}(u)\frac{\partial^{|\alpha|}v}{\partial\nu_{\delta}^{\alpha}}(\delta)$ for $|\alpha|>0$.}
Completely based on the terms in the right hand side of the generalized Green's identity~\eqref{eq:HmGreenIntro}, we define the fully $H^m$-nonconforming virtual element $(K, \mathcal N_k(K), V_k(K))$ for $m>n$.
When $K$ is a simplex and $k=m=n+1$, the degrees of freedom $\mathcal N_k(K)$ are same as those of the nonconforming finite element in \cite{WuXu2019}.
When $K$ is a simplex and $k=m>n=2$, we also recover the the degrees of freedom mentioned in \cite[page~268]{HuZhang2017}.

The serendipity approach in \cite{BeiraodaVeigaBrezziMariniRusso2016c,Russo2016,BeiraoDaVeigaBrezziDassiMariniEtAl2018} is employed to reduce the dimension of the virtual element $(K, \mathcal N_k(K), V_k(K))$ to achieve the serendipity virtual element $(K, \mathcal N_k^s(K), V_k^s(K))$.
\XH{If we choose the degrees of freedom $\mathcal N_k^s(K)$ carefully, the virtual element $(K, \mathcal N_k^s(K), V_k^s(K))$ may reduce to an $H^m$-nonconforming finite element, whose shape functions are polynomials of degree no more than some nonnegative $k_s\leq k$.
To be specific, when $n=2$, $m=3$, $k=5$, $k_s=4$, $K$ is a triangle and choose a reduced degrees of freedom, the serendipity virtual element $(K, \mathcal N_k^s(K), V_k^s(K))$ is reduced to an $H^3$-nonconforming finite element in \cite{HuZhang2017}.
As a result, the techniques in assembling the stiffness matrix of the virtual element methods can be used to assemble the stiffness matrix of the finite element methods in this case. More importantly, this hints a way to recover some existing finite elements and construct new $H^m$-nonconforming finite elements, nevertheless it is not easy to verify the assumption in the serendipity.}

\XH{The local $H^m$ projection $\Pi^K$ is defined in view of the local $m$-harmonic problem.
Adopting the Taylor's theorem, we prove the inverse inequality of polynomials on the general polytope $K$ by only assuming $K$ is star-shaped,
while $K$ admitting a virtual quasi-uniform triangulation is assumed in~\cite{ChenHuang2018}.
According to this inverse inequality, the operator $(-\Delta)^m: \mathbb P_{k}(K)\to\mathbb P_{k-2m}(K)$ is shown to be onto and have a continuous right inverse.
Based on the fact that the operator $(-\Delta)^m: \mathbb P_{k}(K)\to\mathbb P_{k-2m}(K)$ is onto, we propose a stabilization term using only the boundary degrees of freedom,
whereas all the degrees of freedom are involved in the stabilization term in \cite{ChenHuang2019}.
After introducing the discrete right hand side term, we
design the $H^m$-nonconforming virtual element methods to approximate solutions of the $m$-harmonic equation.
For the case $2m \leq k \leq 3m - 2$, we define the right hand side term as $(f, Q_h^{m-1}\Pi_h v_h+Q_h^{k-2m}(v_h-\Pi_h v_h))$, rather than $(f, Q_h^{m-1}v_h)$ in \cite{ChenHuang2019}. The gain of this new right hand side term is that we do not need to modify the virtual element space $V_k(K)$ for $2m \leq k \leq 3m - 2$, whereas the modification of $V_k(K)$ is required in \cite{ChenHuang2019}.}
%It is worth mentioning that the stabilization term in this paper uses less degrees of freedom than those in \cite{ChenHuang2019}, in which all the degrees of freedom are involved in the stabilization term.

We analyze the $H^m$-nonconforming virtual element methods under the assumptions that each element in the mesh $\mathcal T_h$ is star-shaped, and $\mathcal T_h$ admits a virtual quasi-uniform triangulation.
%Adopting the Taylor's theorem, we first prove the inverse inequality of polynomials on the general polytope $K$ by only assuming $K$ is star-shaped,
%while $K$ admitting a virtual quasi-uniform triangulation is assumed in~\cite{ChenHuang2018}.
Applying the bubble function technique, the generalized Green's identity, the Poincar\'e inquality and the trace inequality, the norm equivalence of the stabilization on $\ker(\Pi^K)\cap V_k(K)$ is derived.
As in \cite{ChenHuang2019}, after obtaining a bound on the jump $\llbracket\nabla_h^sv_h\rrbracket$ and the
error estimate of the canonical interpolation,  we achieve
the optimal error estimates of the $H^m$-nonconforming virtual element methods.
We also consider the implementation of the virtual element method.

The rest of this paper is organized as follows. Some notations and mesh conditions are shown in Section 2.
The generalized Green's identity and the fully $H^m$-nonconforming virtual element are presented in Section 3. In Section 4, we propose the $H^m$-nonconforming virtual element method, and prove the norm equivalence and the weak continuity.
We develop the optimal error analysis for the $H^m$-nonconforming virtual element methods in Section 5.
Finally we discuss the implementation of the virtual element method is Section 6.

\section{Preliminaries}

\subsection{Notations}
In this paper we will adopt the same notations as in \cite{ChenHuang2019}.
Let $\Omega\subset \mathbb{R}^n~(n\geq 1)$ be a bounded
polytope.
For any nonnegative integer $r$ and $1\leq\ell\leq n$, notation $\mathbb T_{\ell}(r):=(\mathbb R^{\ell})^r=\prod_{j=1}^{r}\mathbb R^{\ell}$ stands for the set of $r$-tensor spaces over $\mathbb R^{\ell}$.
Given a bounded domain $G\subset\mathbb{R}^{n}$ with $n\in\mathbb N$ and a
non-negative integer $k$, let $H^k(G; \mathbb T_{\ell}(r))$ be the usual Sobolev space of functions
over $G$ taking values in the tensor space $\mathbb T_{\ell}(r)$, whose norm and semi-norm are denoted by
$\Vert\cdot\Vert_{k,G}$ and $|\cdot|_{k,G}$ respectively. Set $H^k(G):=H^k(G; \mathbb T_{\ell}(0))$.
Define $H_0^k(G)$ as the closure of $C_{0}^{\infty}(G)$ with
respect to the norm $\Vert\cdot\Vert_{k,G}$, and define $H_0^1(G; \mathbb T_{\ell}(r))$ in a similar way.
Let $(\cdot, \cdot)_G$ be the standard inner product on $L^2(G; \mathbb T_{\ell}(r))$. If $G$ is $\Omega$, we abbreviate
$\Vert\cdot\Vert_{k,G}$, $|\cdot|_{k,G}$ and $(\cdot, \cdot)_G$ by $\Vert\cdot\Vert_{k}$, $|\cdot|_{k}$ and $(\cdot, \cdot)$,
respectively.
Let $\mathbb P_k(G)$ be the set of all
polynomials over $G$ with the total degree no more than $k$, whose tensorial version space is denoted by $\mathbb P_{k}(G; \mathbb T_{\ell}(r))$. Let $\mathbb P_k(G):=\{0\}$ if $k<0$.
%For any nonnegative integers $j<k$, set
%\[
%\mathbb P_{k,j}^{\perp}(G) :=\{v\in \mathbb P_k(G): (v, q)_G=0\quad\forall~q\in\mathbb P_{j}(G)\},
%\]
%i.e., $\mathbb P_{k,j}^{\perp}(G)\subset\mathbb P_{k}(G)$ is the orthogonal complement space of $\mathbb P_{j}(G)$ of $\mathbb P_{k}(G)$ with respect to the inner product $(\cdot,\cdot)_G$.
Let $Q_k^{G}$ be the $L^2$-orthogonal projection onto $\mathbb P_k(G; \mathbb T_{\ell}(r))$.
For any $F\subset\partial G$,
denote by $\nu_{G, F}$ the
unit outward normal to $\partial G$. Without causing any confusion, we will abbreviate $\nu_{G, F}$ as $\nu$ for simplicity.

For an $n$ dimensional multi-index $\alpha = (\alpha_1, \cdots , \alpha_n)$ with $\alpha_i\in\mathbb Z^+\cup\{0\}$, define $|\alpha|:=\sum_{i=1}^n\alpha_i$.
%We will use $\alpha, \beta$ to denote the multi-indexes.
For $0\leq j \leq n$, let $A_{j}$ be the set consisting of all multi-indexes $\alpha$ with $\sum_{i=j+1}^n\alpha_i=0$, i.e., non-zero index only exists for $1\leq i\leq j$.
For any non-negative integer $\ell$, define the scaled monomial $\mathbb M_{\ell}(G)$ on a $j$-dimensional domain $G$
\[
\mathbb M_{\ell}(G):=\left\{\Big (\frac{\bs x-\bs x_G}{h_G}\Big )^{\alpha}, \alpha\in A_{j}, |\alpha|\leq \ell\right\},
\]
where $h_G$ is the diameter of $G$ and $\bs x_G$ is the centroid of $G$. And $\mathbb M_{\ell}(G):=\emptyset$ if $\ell<0$.
For ease of presentation, let $N_{G,\ell}:=\#\mathbb M_{\ell}(G)$, and all the functions in $\mathbb M_{\ell}(G)$ be \XH{$\{\textsf{m}_{G,i}\}_{i=1}^{N_{G,\ell}}$}.
%For any finite set $\mathcal{S}$ with the cardinality $\#\mathcal{S}$, let $\mathbb{R}(\mathcal{S})$ be a set of functions from $\mathcal{S}$ into $\mathbb{R}$.

Given $r$-tensors $\tau, \varsigma\in \mathbb T_{\ell}(r)$ and a vector $v\in\mathbb R^{\ell}$, define the scalar product $\tau:\varsigma\in \mathbb R$  and the dot product
$\tau\cdot v\in \mathbb T_{\ell}(r-1)$ by (cf. \cite{Schedensack2016})
\[
\tau:\varsigma:=\sum_{(j_1,\cdots,j_{r})\in\{1,\cdots,\ell\}^{r}}\tau_{j_1,\cdots, j_{r}}\varsigma_{j_1,\cdots,j_{r}},
\]
\[
(\tau\cdot v)_{j_1,\cdots, j_{r-1}}:=\sum_{i=1}^{\ell}\tau_{j_1,\cdots, j_{r-1},i}v_i\quad\forall~(j_1,\cdots, j_{r-1})\in\{1,\cdots,\ell\}^{r-1},
\]
which will be abbreviated as $\tau v$.

Let $\{\mathcal {T}_h\}$ be a family of partitions
of $\Omega$ into nonoverlapping simple polytopal elements with $h:=\max_{K\in \mathcal {T}_h}h_K$.
%and $h_K:=\mbox{diam}(K)$.
Let $\mathcal{F}_h^r$ be the set of all $(n-r)$-dimensional faces
of the partition $\mathcal {T}_h$ for $r=1, 2, \cdots, n$, and its boundary part
\[
\mathcal{F}_h^{r, \partial}:=\{F\in\mathcal{F}_h^r: F\subset\partial\Omega\},
\]
and interior part $\mathcal{F}_h^{r, i}:=\mathcal{F}_h^{r}\backslash \mathcal{F}_h^{r, \partial}$.
For simplicity, let $\mathcal{F}_h^0:=\mathcal {T}_h$.
Moreover, we set for each $K\in\mathcal{T}_h$
\[
\mathcal{F}^r(K):=\{F\in\mathcal{F}_h^r: F\subset\partial K\}.%, \quad \mathcal{E}_h^i(K):=\{e\in\mathcal{E}_h^i: e\subset\partial K\},
\]
The supscript $r$ in $\mathcal{F}_h^r$ represents the co-dimension of an $(n-r)$-dimensional face $F$. %as we shall show later the degree of freedom will be associated to the $r$-normal vectors of $F$.
Similarly, we define
\[
\mathcal F^j(F):=\{e\in\mathcal{F}_h^{r+j}: e\subset\overline{F}\}.
\]
Here $j$ is the co-dimension relative to the face $F$. %Apparently $\mathcal F^0(F)=F$.
For any $(n-2)$-dimensional face $e\in\mathcal{F}_h^{2}$,  denote
\[
\partial^{-1}e :=\{F\in\mathcal{F}_h^{1}: e\subset \partial F\}.
\]
Similarly for any $(n-1)$-dimensional face $F\in\mathcal{F}_h^{1}$,  let
\[
\partial^{-1}F :=\{K\in\mathcal{T}_h: F\in \mathcal{F}^1(K)\}.
\]
%For any element $K\in\mathcal{T}_h$, set the patch
%\[
%\omega_K:=\textrm{interior}\Big (\bigcup\{\overline{K'}\in\mathcal T_h: \overline{K'}\cap\overline{K}\neq\varnothing\}\Big ).
%\]

For any $F\in\mathcal{F}_h^r$ with $1\leq r\leq n-1$, let $\nu_{F,1}, \cdots, \nu_{F,r}$ be its
%linearly independent
mutually perpendicular unit normal vectors, and define the surface gradient on $F$ as
\begin{equation}\label{eq:surfacegrad}
\nabla_{F}v:=\nabla v-\sum_{i=1}^r\frac{\partial v}{\partial\nu_{F,i}}\nu_{F,i},
\end{equation}
namely the projection of $\nabla v$ to the face $F$, which is independent of the choice of the normal vectors. And denote by $\div_{F}$ the corresponding surface divergence.
For any $\delta\in\mathcal{F}_h^n$ and $i=1,\cdots, n$, let $\nu_{\delta,i}:=(0,\cdots,0,1,0,\cdots,0)^{\intercal}$ be the $n$-tuple with all components equal to $0$, except the $i$th, which is $1$.
For any $F\in\mathcal F^{r}_h$ and $\alpha\in A_{r}$ for $r=1,\cdots, n$, set
\[
\frac{\partial^{|\alpha|}v}{\partial\nu_{F}^{\alpha}}:=\frac{\partial^{|\alpha|}v}{\partial\nu_{F, 1}^{\alpha_1}\cdots\partial\nu_{F, r}^{\alpha_{r}}}.
\]

For non-negative integers $m$ and $k$, let
\[
H^m(\mathcal T_h):=\{v\in L^2(\Omega): v|_K\in H^m(K)\textrm{ for each } K\in\mathcal T_h\},
\]
\[
\mathbb P_k(\mathcal T_h):=\{v\in L^2(\Omega): v|_K\in \mathbb P_k(K)\textrm{ for each } K\in\mathcal T_h\}.
\]
For a function $v\in H^m(\mathcal T_h)$, equip the usual broken $H^m$-type norm and semi-norm
\[
\|v\|_{m,h}:=\Big (\sum_{K\in\mathcal T_h}\|v\|_{m,K}^2\Big )^{1/2},\quad |v|_{m,h}:=\Big (\sum_{K\in\mathcal T_h}|v|_{m,K}^2\Big )^{1/2}.
\]
For any $K\in\mathcal T_h$, $\delta\in \mathcal F^n(K)$, and any function $v$ defined on $K$,
we will rewrite $v(\bs x_{\delta})$ as $v(\delta)$ for simplicity.

We introduce jumps on ($n-1$)-dimensional faces.
Consider two adjacent elements $K^+$ and $K^-$ sharing an interior ($n-1$)-dimensional face $F$.
Denote by $\nu^+$ and $\nu^-$ the unit outward normals
to the common face $F$ of the elements $K^+$ and $K^-$, respectively.
For a scalar-valued or tensor-valued function $v$, write $v^+:=v|_{K^+}$ and $v^-
:=v|_{K^-}$.   Then define the jump on $F$ as
follows:
\[
\llbracket v\rrbracket:=v^+\nu_{F,1}\cdot \nu^++v^-\nu_{F,1}\cdot \nu^-.
\]
On a face $F$ lying on the boundary $\partial\Omega$, the above term is
defined by
\[
\llbracket v\rrbracket
   :=v\nu_{F,1}\cdot \nu.
\]

\subsection{Mesh conditions}
We impose the following conditions on the mesh $\mathcal T_h$.
\begin{itemize}
 \item[(A1)] Each element $K\in \mathcal T_h$ and each face $F\in \mathcal F_h^r$ for $1\leq r\leq n-1$ is star-shaped with a uniformly bounded chunkiness parameter.

 \item[(A2)] There exists a quasi-uniform simplicial mesh $\mathcal T_h^*$ such that each $K\in \mathcal T_h$ is a union of some simplexes in $\mathcal T_h^*$.
\end{itemize}

Throughout this paper, we also use
``$\lesssim\cdots $" to mean that ``$\leq C\cdots$", where
$C$ is a generic positive constant independent of mesh size $h$, but may depend on the chunkiness parameter of the polytope, the degree of polynomials $k$, the order of differentiation $m$, the dimension of space $n$, and the shape regularity and quasi-uniform constants of the virtual triangulation $\mathcal T^*_h$,
which may take different values at different appearances. And $A\eqsim B$ means $A\lesssim B$ and $B\lesssim A$.
Hereafter, we always assume $k\geq m$.

Note that (A1) and (A2) imply ${\rm diam}(F) \eqsim {\rm diam} (K)$ for all $F\in \mathcal F^r(K), 1\leq r\leq n-1$.
For a star-shaped domain $D$, the following trace inequality of $H^1(D)$ holds \cite[(2.18)]{BrennerSung2018}
\begin{equation}\label{L2trace}
\|v\|_{0,\partial D}^2 \lesssim h_D^{-1}\|v\|_{0,D}^2 + h_D |v|_{1,D}^2 \quad \forall~v\in H^1(D).
\end{equation}
When $D\subset\mathbb R$, the notation $\|v\|_{0,\partial D}$ means $\|v\|_{L^{\infty}(\partial D)}$.
%And we also have the following Sobolev inequality \cite[(2.3)]{BrennerSung2018}
%\begin{equation}\label{eq:Sobolevinequality}
%\|v\|_{L^{\infty}(F)}\lesssim h_F^{-1}\|v\|_{0,F} + |v|_{1,F} + h_F|v|_{2,F}\quad\forall~v\in H^2(F), F\in \mathcal F^{n-2}(K).
%\end{equation}

\section{$H^m$-Nonconforming Virtual Element with $m>n$}

In this section, we will construct the $H^m$-nonconforming virtual element with integer $m>n\geq1$.
For any scalar or tensor-valued smooth function $v$, nonnegative integer $j$, $F\in \mathcal{F}_h^r$ with $1\leq r\leq n$, and $\alpha\in A_{r}$, we use $D^j_{F, \alpha}(v)$ to denote some $j$-th order derivative of $v$ restrict on $F$, which may take different expressions at different appearances.
%And for any $\delta\in \mathcal{F}_h^n$ and nonnegative integer $\ell$, let $D^j_{\delta, \ell}(v)\in \mathbb T_{n}(\ell)$
%be an $\ell$-tensor at $\delta$ whose component is some $j$-th order derivative of $v$.

For $n=1$ and any $e\in\mathcal T_h$, applying the integration by parts in one dimension, we have for any $u\in H^{2m}(e)$ and $v\in H^m(e)$
\[
\big(u^{(m)}, v^{(m)}\big)_e=(-1)^m\big(u^{(2m)}, v\big)_e + \sum_{\delta\in\mathcal F^1(e)}\sum_{i=0}^{m-1}(-1)^iu^{(2m-1-i)}(\delta)v^{(i)}(\delta)\nu_{e,\delta}.
\]
This is just the Green's identity in one dimension. Here $v^{(i)}$ means the $i$-th order derivative of $v$, and
\[
\nu_{e,\delta}=
\begin{cases}
1, & \textrm{ if } \delta \textrm{ is the right end point of } e, \\
-1, & \textrm{ if } \delta \textrm{ is the left end point of } e.
\end{cases}
\]

\subsection{Generalized Green's identity in two dimensions}

Then consider the generalized Green's identity in two dimensions in this subsection, i.e. $n=2$.

For each $e\in\mathcal{F}_h^1$, denote by $t_{e}$ the unit tangent vector, which will be also represented by $\nu_{e,2}$ for ease of presentation.
Let $\mathfrak{S}_{\ell}$ be the set of all permutations of $(1, 2, \cdots, \ell)$ for each positive integer $\ell$.
For $i=0, 1,\cdots, \ell$, define a set
\begin{align*}
\mathfrak{S}(\ell, i):=\{(j_1,\cdots,j_{\ell}):& \textrm{ there exists } \sigma\in\mathfrak{S}_{\ell} \textrm{ such that }\\ &\;j_{\sigma(1)}=\cdots=j_{\sigma(i)}=1,\;\;  j_{\sigma(i+1)}=\cdots=j_{\sigma(\ell)}=2\}.
\end{align*}
Apparently
\[
\mathfrak{S}(\ell, 0)=\{(2,\cdots, 2)\},\quad \mathfrak{S}(\ell, \ell)=\{(1,\cdots, 1)\}.
\]
For $0\leq i\leq\ell\leq r$, any $\tau\in\mathbb T_{2}(r)$ and $\sigma=(j_1,\cdots,j_{\ell})\in \mathfrak{S}(\ell, i)$, define $\tau\nu_e^{\sigma}\in \mathbb T_{2}(r-\ell)$ as
\[
\tau\nu_e^{\sigma}:=\tau\nu_{e,j_1}\cdots\nu_{e,j_{\ell}}.%\tau\nu_{e,\sigma(1)}\cdots\nu_{e,\sigma(\ell)}.
\]
We also use $\tau\nu_e^{\sigma}$ to mean $\tau$ when $\ell=0$ for ease of presentation.

\begin{lemma}\label{lemma:ibpedge}
Let $K\in\mathcal T_h$, $e\in\mathcal{F}^1(K)$ and $s$ be a positive integer. It holds for any $\tau\in H^{s}(e;\mathbb T_2(s))$ and $(\nabla^sv)|_e\in L^2(e;\mathbb T_2(s))$
\begin{align*}
(\tau,\nabla^sv)_e=&\sum_{j=0}^{s}\sum_{\sigma\in\mathfrak{S}(s, j)}(-1)^{s-j}\left(\frac{\partial^{s-j}(\tau\nu_e^{\sigma})}{\partial{t_e^{s-j}}}, \frac{\partial^{j}v}{\partial{\nu_{e,1}^{j}}}\right)_e \\
& + \sum_{j=0}^{s-1}\sum_{\ell=0}^{j}\sum_{\sigma\in\mathfrak{S}(s, \ell)}\sum_{\delta\in\partial e}\nu_{e,\delta}(-1)^{s-1-j}\frac{\partial^{s-1-j}(\tau\nu_e^{\sigma})}{\partial{t_e^{s-1-j}}}(\delta) \frac{\partial^{j}v}{\partial{t_e^{j-\ell}}\partial{\nu_{e,1}^{\ell}}}(\delta).
\end{align*}
\end{lemma}
\begin{proof}
It follows from the integration by parts
\begin{align*}
(\tau,\nabla^sv)_e=&\left(\tau,\left(\nu_{e,1}\frac{\partial}{\partial{\nu_{e,1}}} + t_e\frac{\partial}{\partial{t_e}}\right)^sv\right)_e = \sum_{\ell=0}^s\sum_{\sigma\in\mathfrak{S}(s, \ell)}\left(\tau\nu_e^{\sigma}, \frac{\partial^sv}{\partial{t_e^{s-\ell}}\partial{\nu_{e,1}^{\ell}}}\right)_e \\
= &\sum_{\sigma\in\mathfrak{S}(s, s)}\left(\tau\nu_e^{\sigma}, \frac{\partial^sv}{\partial{\nu_{e,1}^{s}}}\right)_e -\sum_{\ell=0}^{s-1}\sum_{\sigma\in\mathfrak{S}(s, \ell)}\left(\frac{\partial(\tau\nu_e^{\sigma})}{\partial{t_e}}, \frac{\partial^{s-1}v}{\partial{t_e^{s-1-\ell}}\partial{\nu_{e,1}^{\ell}}}\right)_e \\
& + \sum_{\ell=0}^{s-1}\sum_{\sigma\in\mathfrak{S}(s, \ell)}\sum_{\delta\in\partial e}\nu_{e,\delta}(\tau\nu_e^{\sigma})(\delta) \frac{\partial^{s-1}v}{\partial{t_e^{s-1-\ell}}\partial{\nu_{e,1}^{\ell}}}(\delta).  \end{align*}
Applying the integration by parts to the second term of the right hand side, we get
\begin{align*}
(\tau,\nabla^sv)_e=&\sum_{j=s-2}^{s}\sum_{\sigma\in\mathfrak{S}(s, j)}(-1)^{s-j}\left(\frac{\partial^{s-j}(\tau\nu_e^{\sigma})}{\partial{t_e^{s-j}}}, \frac{\partial^{j}v}{\partial{\nu_{e,1}^{j}}}\right)_e \\
&+\sum_{\ell=0}^{s-3}\sum_{\sigma\in\mathfrak{S}(s, \ell)}\left(\frac{\partial^2(\tau\nu_e^{\sigma})}{\partial{t_e^2}}, \frac{\partial^{s-2}v}{\partial{t_e^{s-2-\ell}}\partial{\nu_{e,1}^{\ell}}}\right)_e \\
& + \sum_{j=s-2}^{s-1}\sum_{\ell=0}^{j}\sum_{\sigma\in\mathfrak{S}(s, \ell)}\sum_{\delta\in\partial e}\nu_{e,\delta}(-1)^{s-1-j}\frac{\partial^{s-1-j}(\tau\nu_e^{\sigma})}{\partial{t_e^{s-1-j}}}(\delta) \frac{\partial^{j}v}{\partial{t_e^{j-\ell}}\partial{\nu_{e,1}^{\ell}}}(\delta).
\end{align*}
Along this way, we can finish the proof by applying the integration by parts recursively.
\end{proof}

\begin{lemma}
Let $K\in\mathcal T_h$ and integer $s\geq n=2$.
There exist differential operators $D^{s-1-|\alpha|}_{e, \alpha}$ for $e\in \mathcal F^1(K)$ and $\alpha\in A_{1}$ with $|\alpha|\leq s-1$, and $D^{s-2-|\alpha|}_{\delta, \alpha}$ for $\delta\in \mathcal F^2(K)$ and $\alpha\in A_{2}$ with $|\alpha|\leq s-2$ such that for any $\tau\in H^{s}(K;\mathbb T_2(s))$ and $v\in H^{s}(K)$, it holds
\begin{align}
(\tau, \nabla^sv)_K = &\left((-\div)^s\tau,v\right)_K+ \sum_{e\in\mathcal F^1(K)}\sum_{\alpha\in A_{1}\atop|\alpha|\leq s-1}\Big ( D^{s-1-|\alpha|}_{e, \alpha}(\tau), \;\frac{\partial^{|\alpha|}v}{\partial\nu_{e}^{\alpha}}\Big )_e \notag\\
 & + \sum_{\delta\in\mathcal F^2(K)}\sum_{\alpha\in A_{2}\atop|\alpha|\leq s-2}D^{s-2-|\alpha|}_{\delta, \alpha}(\tau)\frac{\partial^{|\alpha|}v}{\partial\nu_{\delta}^{\alpha}}(\delta). \label{eq:HsGreen2d}
\end{align}
\end{lemma}
\begin{proof}
Due to the integration by parts, we get
\begin{align}
(\tau,\nabla^sv)_K&=-(\div\tau,\nabla^{s-1}v)_K + \sum_{e\in\mathcal F^1(K)}(\tau \nu_{K,e},\nabla^{s-1}v)_{e} \notag\\
&=\left((-\div)^s\tau,v\right)_K + \sum_{e\in\mathcal F^1(K)}\sum_{i=0}^{s-1}(-1)^{s-1-i}\left((\div^{s-1-i}\tau)\nu_{K,e}, \nabla^{i}v\right)_{e}. \label{eq:20190217}
\end{align}
%For $e\in\mathcal E_h(K)$ and $i=1,\cdots, m-1$, it follows from Lemma~\ref{lemma:ibpedge}
%\begin{align*}
%&(-1)^{m-1-i}\left((\div^{m-1-i}\tau)n_K, \nabla^{i}v\right)_{e} \\
%=&\sum_{j=0}^{i-1}\sum_{\ell=0}^{j}\sum_{\sigma\in\mathfrak{S}(i, \ell)}\sum_{\delta\in\partial e}\varepsilon(e,\delta)(-1)^{m-2-j}\frac{\partial^{i-1-j}((\div^{m-1-i}\tau)n_K\nu_e^{\sigma})}{\partial{t_e^{i-1-j}}}(\delta) \frac{\partial^{j}v}{\partial{t_e^{j-\ell}}\partial{n_e^{\ell}}}(\delta) \\
%&+  \sum_{j=0}^{i}\sum_{\sigma\in\mathfrak{S}(i, j)}(-1)^{m-1-j}\left(\frac{\partial^{i-j}((\div^{m-1-i}\tau)n_K\nu_e^{\sigma})}{\partial{t_e^{i-j}}}, \frac{\partial^{j}v}{\partial{n_e^{j}}}\right)_e.
%\end{align*}
Then it follows from Lemma~\ref{lemma:ibpedge}
\[
\begin{aligned}
&(\tau,\nabla^sv)_K-\left((-\div)^s\tau,v\right)_K\\
=&\sum_{e\in\mathcal F^1(K)}\sum_{i=1}^{s-1}\sum_{j=0}^{i-1}\sum_{\ell=0}^{j}\sum_{\sigma\in\mathfrak{S}(i, \ell)}\sum_{\delta\in\partial e}D_{\delta,\sigma}^{s-2-j}(\tau) \frac{\partial^{j}v}{\partial{t_e^{j-\ell}}\partial{\nu_{e,1}^{\ell}}}(\delta) \\
&+ \sum_{e\in\mathcal F^1(K)}\sum_{i=0}^{s-1} \sum_{j=0}^{i}\sum_{\sigma\in\mathfrak{S}(i, j)}\Big(D_{e,\sigma}^{s-1-j}(\tau), \frac{\partial^{j}v}{\partial{\nu_{e,1}^{j}}}\Big)_e \\
=& \sum_{e\in\mathcal F^1(K)}\sum_{j=0}^{s-2}\sum_{i=j+1}^{s-1}\sum_{\ell=0}^{j}\sum_{\sigma\in\mathfrak{S}(i, \ell)}\sum_{\delta\in\partial e}D_{\delta,\sigma}^{s-2-j}(\tau) \frac{\partial^{j}v}{\partial{t_e^{j-\ell}}\partial{\nu_{e,1}^{\ell}}}(\delta) \\
&+ \sum_{e\in\mathcal F^1(K)} \sum_{j=0}^{s-1}\sum_{i=j}^{s-1}\sum_{\sigma\in\mathfrak{S}(i, j)}\Big(D_{e,\sigma}^{s-1-j}(\tau), \frac{\partial^{j}v}{\partial{\nu_{e,1}^{j}}}\Big)_e,
\end{aligned}
\]
where
\[
D_{\delta,\sigma}^{s-2-j}(\tau)=(-1)^{s-2-j}\frac{\partial^{i-1-j}((\div^{s-1-i}\tau)\nu_{K,e}\nu_e^{\sigma})}{\partial{t_e^{i-1-j}}}(\delta)\nu_{e,\delta},
\]
\[
D_{e,\sigma}^{s-1-j}(\tau)=(-1)^{s-1-j}\frac{\partial^{i-j}((\div^{s-1-i}\tau)\nu_{K,e}\nu_e^{\sigma})}{\partial{t_e^{i-j}}}.
\]
This indicates \eqref{eq:HsGreen2d}.
\end{proof}

As an immediate result of \eqref{eq:HsGreen2d}, we achieve the generalized Green's identity in two dimensions as follows.
\begin{lemma}\label{lem:greenidentity2d}
Let $K\in\mathcal T_h$ and integer $m> n=2$. There exist differential operators $D^{2m-1-|\alpha|}_{e, \alpha}$ for $e\in \mathcal F^1(K)$ and $\alpha\in A_{1}$ with $|\alpha|\leq m-1$, and $D^{2m-2-|\alpha|}_{\delta, \alpha}$ for $\delta\in \mathcal F^2(K)$ and $\alpha\in A_{2}$ with $|\alpha|\leq m-2$ such that for any $u\in H^{2m}(K)$ and $v\in H^{m}(K)$, it holds
\begin{align}
(\nabla^mu, \nabla^mv)_K = &\left((-\Delta)^mu,v\right)_K+ \sum_{e\in\mathcal F^1(K)}\sum_{\alpha\in A_{1}\atop|\alpha|\leq m-1}\Big ( D^{2m-1-|\alpha|}_{e, \alpha}(u), \;\frac{\partial^{|\alpha|}v}{\partial\nu_{e}^{\alpha}}\Big )_e \notag\\
 &+ \sum_{\delta\in\mathcal F^2(K)}\sum_{\alpha\in A_{2}\atop|\alpha|\leq m-2}D^{2m-2-|\alpha|}_{\delta, \alpha}(u)\frac{\partial^{|\alpha|}v}{\partial\nu_{\delta}^{\alpha}}(\delta). \label{eq:HmGreen2d}
\end{align}
\end{lemma}

\subsection{Generalized Green's identity in $n$ dimensions}

Now we extend
Lemma~\ref{lemma:ibpedge} and Lemma~\ref{lem:greenidentity2d}
to any dimension. \XH{To this end, we recall two results in~\cite{ChenHuang2019}.
\begin{lemma}[Lemma~3.1 in \cite{ChenHuang2019}]\label{lem:20190605-1Pre}
%Let $K\in\mathcal T_h$, positive integer $s\leq m-1$ and $F\in \mathcal F^r(K)$ with $1\leq r\leq n-s$.
Let $K\in\mathcal T_h$, $F\in \mathcal F^r(K)$ with $1\leq r\leq n-1$, and $s$ be a positive integer satisfying $s\leq n-r$.
There exist differential operators $D^{s-j-|\alpha|}_{e, \alpha}$ for $j=0,\cdots, s$, $e\in \mathcal F^j(F)$ and $\alpha\in A_{r+j}$ with $|\alpha|\leq s-j$ such that for any $\tau\in H^{s}(F; \mathbb T_{n}(s))$ and $(\nabla^sv)|_F\in L^2(F; \mathbb T_{n}(s))$, it holds
\begin{equation}\label{eq:HmfaceGreenPre}
(\tau, \nabla^sv)_F = \sum_{j=0}^s\sum_{e\in\mathcal F^j(F)}\sum_{\alpha\in A_{r+j}\atop|\alpha|\leq s-j}\Big ( D^{s-j-|\alpha|}_{e, \alpha}(\tau), \;\frac{\partial^{|\alpha|}v}{\partial\nu_{e}^{\alpha}}\Big )_e.
\end{equation}
\end{lemma}
Another one is the recurrence relation derived in the proof of Lemma~3.1 in \cite{ChenHuang2019}
\begin{equation}\label{eq:20200123}
(\tau, \nabla^{\ell+1} v)_F=\sum_{i=1}^r(\tau\nu_{F,i}, \nabla^{\ell}\frac{\partial v}{\partial\nu_{F,i}})_F - (\div_F\tau, \nabla^{\ell}v)_F +\sum_{e\in\mathcal F^1(F)}(\tau\nu_{F,e}, \nabla^{\ell}v)_e
\end{equation}
for any positive integer $\ell$, and $F\in \mathcal F^r(K)$ with $1\leq r\leq n-1$.}

\begin{lemma}%\label{lem:ibp}
Let $K\in\mathcal T_h$, $F\in \mathcal F^r(K)$ with $1\leq r\leq n-1$, and positive integer $s\geq n-r$. There exist differential operators $D^{s-j-|\alpha|}_{e, \alpha}$ for $j=0,\cdots, n-r-1$, $e\in \mathcal F^j(F)$ and $\alpha\in A_{r+j}$ with $|\alpha|\leq s-j$, and differential operators $D^{s+r-n-|\alpha|}_{\delta, \alpha}$ for $\delta\in\mathcal F^{n-r}(F)$ and $\alpha\in A_{n}$ with $|\alpha|\leq s+r-n$ such that for any $\tau\in H^{s}(F; \mathbb T_{n}(s))$ and $(\nabla^sv)|_F\in L^2(F; \mathbb T_{n}(s))$, it holds
\begin{align}
(\tau, \nabla^sv)_F =& \sum_{j=0}^{n-r-1}\sum_{e\in\mathcal F^j(F)}\sum_{\alpha\in A_{r+j}\atop|\alpha|\leq s-j}\Big ( D^{s-j-|\alpha|}_{e, \alpha}(\tau), \;\frac{\partial^{|\alpha|}v}{\partial\nu_{e}^{\alpha}}\Big )_e  \notag\\
 &+ \sum_{\delta\in\mathcal F^{n-r}(F)}\sum_{\alpha\in A_{n}\atop|\alpha|\leq s+r-n}D^{s+r-n-|\alpha|}_{\delta, \alpha}(\tau)\frac{\partial^{|\alpha|}v}{\partial\nu_{\delta}^{\alpha}}(\delta).\label{eq:HsfaceGreen}
\end{align}
\end{lemma}
\begin{proof}
\XH{The identities \eqref{eq:HsfaceGreen} and \eqref{eq:HmfaceGreenPre} are same for $s=n-r$.}
%The identity \eqref{eq:HsfaceGreen} has been proved for $s=n-r$ in \cite{ChenHuang2019}.
Assume the identity~\eqref{eq:HsfaceGreen} holds for $s=\ell$ with integer $\ell\geq n-r$, then let us prove it is also true for $s=\ell+1$.
Applying \eqref{eq:HsfaceGreen} with $s=\ell$ to each term in the right hand side of \eqref{eq:20200123}, we have
%for $r=1,2,\cdots, n-1$
\begin{align*}
(\tau\nu_{F,i}, \nabla^{\ell}\frac{\partial v}{\partial\nu_{F,i}})_F
&=\sum_{j=0}^{n-r-1}\sum_{e\in\mathcal F^j(F)}\sum_{\alpha\in A_{r+j}\atop|\alpha|\leq {\ell}-j}\Big (D^{\ell-j-|\alpha|}_{e, \alpha}(\tau\nu_{F,i}), \;\frac{\partial^{|\alpha|}}{\partial\nu_{e}^{\alpha}}\Big (\frac{\partial v}{\partial\nu_{F,i}}\Big )\Big )_e \\
&\quad+ \sum_{\delta\in\mathcal F^{n-r}(F)}\sum_{\alpha\in A_{n}\atop|\alpha|\leq \ell+r-n}D^{\ell+r-n-|\alpha|}_{\delta, \alpha}(\tau\nu_{F,i})\frac{\partial^{|\alpha|}}{\partial\nu_{\delta}^{\alpha}}\Big (\frac{\partial v}{\partial\nu_{F,i}}\Big )(\delta),
\end{align*}
\begin{align*}
 (\div_F\tau, \nabla^{\ell}v)_F&=\sum_{j=0}^{n-r-1}\sum_{e\in\mathcal F^j(F)}\sum_{\alpha\in A_{r+j}\atop|\alpha|\leq {\ell}-j}\Big (D^{\ell-j-|\alpha|}_{e, \alpha}(\div_F\tau), \;\frac{\partial^{|\alpha|}v}{\partial\nu_{e}^{\alpha}}\Big )_e \\
 &\quad + \sum_{\delta\in\mathcal F^{n-r}(F)}\sum_{\alpha\in A_{n}\atop|\alpha|\leq \ell+r-n}D^{\ell+r-n-|\alpha|}_{\delta, \alpha}(\div_F\tau)\frac{\partial^{|\alpha|}v}{\partial\nu_{\delta}^{\alpha}}(\delta),
\end{align*}
\begin{align*}
(\tau\nu_{F,e}, \nabla^{\ell}v)_e&=\sum_{j=0}^{n-r-2}\sum_{\tilde e\in\mathcal F^j(e)}\sum_{\alpha\in A_{r+1+j}\atop|\alpha|\leq {\ell}-j}\Big (D^{\ell-j-|\alpha|}_{\tilde e, \alpha}(\tau\nu_{F,e}), \;\frac{\partial^{|\alpha|}v}{\partial\nu_{\tilde e}^{\alpha}}\Big )_{\tilde e} \\
&\quad+ \sum_{\delta\in\mathcal F^{n-r-1}(e)}\sum_{\alpha\in A_{n}\atop|\alpha|\leq \ell+r+1-n}D^{\ell+r+1-n-|\alpha|}_{\delta, \alpha}(\tau\nu_{F,e})\frac{\partial^{|\alpha|}v}{\partial\nu_{\delta}^{\alpha}}(\delta).
\end{align*}
Hence we conclude~\eqref{eq:HsfaceGreen} for $s=\ell+1$ by combining the last fourth equations.
Finally we ends the proof based on the mathematical induction.
\end{proof}

\begin{lemma}%\label{lem:ibp}
Let $K\in\mathcal T_h$, and positive integer $m> n$. There exist differential operators $D^{2m-j-|\alpha|}_{F, \alpha}$ for $j=1,\cdots, n-1$, $F\in \mathcal F^j(K)$ and $\alpha\in A_{j}$ with $|\alpha|\leq m-j$, and differential operators $D^{2m-n-|\alpha|}_{\delta, \alpha}$ for $\delta\in\mathcal F^{n}(K)$ and $\alpha\in A_{n}$ with $|\alpha|\leq m-n$ such that for any $u\in H^{2m}(K)$ and $v\in H^{m}(K)$, it holds
\begin{align}
(\nabla^mu, \nabla^mv)_K =& \left((-\Delta)^mu,v\right)_K+ \sum_{j=1}^{n-1}\sum_{F\in\mathcal F^j(K)}\sum_{\alpha\in A_{j}\atop|\alpha|\leq m-j}\Big ( D^{2m-j-|\alpha|}_{F, \alpha}(u), \;\frac{\partial^{|\alpha|}v}{\partial\nu_{F}^{\alpha}}\Big )_F  \notag\\
 &+ \sum_{\delta\in\mathcal F^{n}(K)}\sum_{\alpha\in A_{n}\atop|\alpha|\leq m-n}D^{2m-n-|\alpha|}_{\delta,\alpha}(u)\frac{\partial^{|\alpha|}v}{\partial\nu_{\delta}^{\alpha}}(\delta).\label{eq:HmGreen}
\end{align}
\end{lemma}
\begin{proof}
It is sufficient to prove that there exist differential operators $D^{m-j-|\alpha|}_{F, \alpha}$ for $j=1,\cdots, n-1$, $F\in \mathcal F^j(K)$ and $\alpha\in A_{j}$ with $|\alpha|\leq m-j$, and differential operators $D^{m-n-|\alpha|}_{\delta, \alpha}$ for $\delta\in\mathcal F^{n}(K)$ and $\alpha\in A_{n}$ with $|\alpha|\leq m-n$ such that for any
$\tau\in H^{m}(K;\mathbb T_n(m))$ and $v\in H^{m}(K)$, it holds
\begin{align}
(\tau, \nabla^mv)_K =& \left((-\div)^m\tau,v\right)_K+ \sum_{j=1}^{n-1}\sum_{F\in\mathcal F^j(K)}\sum_{\alpha\in A_{j}\atop|\alpha|\leq m-j}\Big ( D^{m-j-|\alpha|}_{F, \alpha}(\tau), \;\frac{\partial^{|\alpha|}v}{\partial\nu_{F}^{\alpha}}\Big )_F  \notag\\
 &+ \sum_{\delta\in\mathcal F^{n}(K)}\sum_{\alpha\in A_{n}\atop|\alpha|\leq m-n}D^{m-n-|\alpha|}_{\delta,\alpha}(\tau)\frac{\partial^{|\alpha|}v}{\partial\nu_{\delta}^{\alpha}}(\delta).\label{eq:HmGreentau}
\end{align}
%Employing the integration by parts, we get
%\begin{align*}
%(\tau, \nabla^mv)_K & = ((-\div)\tau, \nabla^{m-1}v)_K + \sum_{F\in\mathcal F^1(K)}(\tau\nu_{K,F},\nabla^{m-1}v)_F \\
%& = ((-\div)^m\tau, v)_K + \sum_{F\in\mathcal F^1(K)}\sum_{j=1}^m\left(((-\div)^{j-1}\tau)\nu_{K,F},\nabla^{m-j}v\right)_F.
%\end{align*}
As \eqref{eq:20190217}, we get from the integration by parts
\[
(\tau, \nabla^mv)_K = ((-\div)^m\tau, v)_K + \sum_{F\in\mathcal F^1(K)}\sum_{j=1}^m\left(((-\div)^{j-1}\tau)\nu_{K,F},\nabla^{m-j}v\right)_F.
\]
Therefore \eqref{eq:HmGreentau} follows from \XH{\eqref{eq:HmfaceGreenPre}} and \eqref{eq:HsfaceGreen}.
\end{proof}

\subsection{Virtual element space}
Inspired by identity~\eqref{eq:HmGreen}, for any element $K\in\mathcal T_h$ and integer $k\geq m$,
the local degrees of freedom $\mathcal N_k(K)$ are given as follows:
\begin{align}
h_K^j(\nabla^jv)(\delta) & \quad\forall~\delta\in\mathcal F^{n}(K), \;j=0,1,\cdots,m-n,\label{Hmdof3}\\
\frac{1}{|F|^{(n-j-|\alpha|)/(n-j)}}(\frac{\partial^{|\alpha|}v}{\partial\nu_{F}^{\alpha}}, q)_F & \quad\forall~q\in\mathbb M_{k-(2m-j-|\alpha|)}(F), F\in\mathcal F^{j}(K), \label{Hmdof2}\\
&\quad j=1,\cdots, n-1, \alpha\in A_{j} \textrm{ with } |\alpha|\leq m-j, \notag\\
\frac{1}{|K|}(v, q)_K & \quad\forall~q\in\mathbb M_{k-2m}(K). \label{Hmdof1}
\end{align}
%on each $F\in\mathcal F^{j}(K)$, where $j=1,\cdots, n-1$, $\alpha\in A_{j}$ and $|\alpha|\leq m-j$.
We will use $\chi_{j,i}^{F,\alpha}$ to denote the degrees of freedom~\eqref{Hmdof2} for simplicity, where $i=1,\cdots,N_{F,k-(2m-j-|\alpha|)}$. %with $N_{F,\alpha}:=\dim(\mathbb M_{k-(2m-j-|\alpha|)}(F))$.

According to the first terms in the inner products of the right hand side of~\eqref{eq:HmGreen}, and the degrees of freedom~\eqref{Hmdof2}-\eqref{Hmdof1}, it is inherent to define the local space of the $H^m$-nonconforming virtual element as
\begin{align*}
V_k(K):=\{ u\in H^m(K):& (-\Delta)^m u\in \mathbb P_{k-2m}(K),   \\
& D^{2m-j-|\alpha|}_{F, \alpha}(u)|_F\in\mathbb P_{k-(2m-j-|\alpha|)}(F)\quad\forall~ F\in\mathcal F^{j}(K),\\
 &\qquad\qquad\quad j=1,\cdots, n-1, \,\alpha\in A_{j} \textrm{ and } |\alpha|\leq m-j\}.
\end{align*}

Combining \XH{Lemma~\ref{lem:20190605-1Pre}}, \eqref{eq:HsfaceGreen} and the definition of the degrees of freedom~\eqref{Hmdof2} yields the following property.
\begin{lemma}\label{lem:bimgunisolvence}
Let $K\in\mathcal T_h$, $F\in \mathcal F^r(K)$ with $1\leq r\leq n-1$, $s=n-r,\cdots, m-r$ satisfying $k\geq 2m-(r+s)$.
For any $\tau\in \mathbb P_{k-(2m-r-s)}(F; \mathbb T_{n}(s))$ and $(\nabla^sv)|_F\in L^2(F; \mathbb T_{n}(s))$, the term
\begin{equation*}
(\tau, \nabla^sv)_F
\end{equation*}
is uniquely determined by the degrees of freedom \eqref{Hmdof2} for all nonnegative integer $j\leq n-r-1$, $e\in\mathcal F^j(F)$, $\alpha\in A_{r+j}$ with $|\alpha|\leq s-j$, and \eqref{Hmdof3}
for all $\delta\in\mathcal F^{n-r}(F)$ and nonnegative integer $j\leq s+r-n$.
When $s<n-r$, the term
\begin{equation*}
(\tau, \nabla^sv)_F
\end{equation*}
is uniquely determined by the degrees of freedom \eqref{Hmdof2} for all nonnegative integer $j\leq s$, $e\in\mathcal F^j(F)$, $\alpha\in A_{r+j}$ with $|\alpha|\leq s-j$.
%is uniquely determined by the degrees of freedom \eqref{Hmdof2}-\eqref{Hmdof3}.
\end{lemma}

Employing the same argument as in the proof of Lemma~3.5 in \cite{ChenHuang2019}, we get from the generalized Green's identity \eqref{eq:HmGreen} and Lemma~\ref{lem:bimgunisolvence} that the degrees of freedom~\eqref{Hmdof3}-\eqref{Hmdof1} are unisolvent for the local virtual element space $V_k(K)$.

\begin{remark}\rm
When $n=1$, for any element $e\in\mathcal T_h$ and integer $k\geq m$,
the local degrees of freedom \eqref{Hmdof3}-\eqref{Hmdof1} will be reduced to
\begin{align*}
v^{(j)}(\delta) & \quad\forall~\delta\in\mathcal F^{1}(e), \;j=0,1,\cdots,m-1, \\
\frac{1}{|e|}(v, q)_e & \quad\forall~q\in\mathbb M_{k-2m}(e).
\end{align*}
And the shape function space will be
\[
V_k(e)=\big\{ v\in H^m(e): v^{(2m)}\in \mathbb P_{k-2m}(e)\big\}=\begin{cases}
\mathbb P_{k}(e), & k\geq 2m,\\
\mathbb P_{2m-1}(e), & k<2m.
\end{cases}
\]
Thus the $H^m$-nonconforming virtual element of order $k$ in one dimension is exactly the $C^{m-1}$-continuous finite element, whose shape functions are polynomials of degree $\max\{k, 2m-1\}$.
\end{remark}

\begin{remark}\rm
When $n=2$, for any element $K\in\mathcal T_h$ and integer $k< 2m$,
the local degrees of freedom \eqref{Hmdof3}-\eqref{Hmdof1} will be reduced to
\begin{align}
h_K^j(\nabla^jv)(\delta) & \quad\forall~\delta\in\mathcal F^{2}(K), \;j=0,1,\cdots,m-2, \label{Hmdof0-kmn2}\\
|e|^{j-1}(\frac{\partial^{j}v}{\partial\nu_{e,1}^{j}}, q)_e & \quad\forall~q\in\mathbb M_{j+1-(2m-k)}(e), e\in\mathcal F^{1}(K), j=0,1,\cdots, m-1.\label{Hmdof1-kmn2}
\end{align}
And the shape function space will be
\begin{align*}
V_k(K)=\big\{ v\in H^m(K): &\, (-\Delta)^m v=0, \,D^{2m-1-j}_{e, \alpha}(v)|_e\in\mathbb P_{j+1-(2m-k)}(e) \textrm{ for each } \\
& \;\;e\in\mathcal F^{1}(K), \textrm{ where } \alpha=(j,0) \textrm{ with } j=0,1,\cdots, m-1\big\}.
\end{align*}
If each element $K\in \mathcal T_h$ is a simplex and $k=m>2$, the degrees of freedom \eqref{Hmdof0-kmn2}-\eqref{Hmdof1-kmn2} are same as those mentioned in \cite[page~268]{HuZhang2017}.
\end{remark}

Hereafter we always assume $n\geq 2$.

\begin{remark}\rm
When $k=m$, the local degrees of freedom~\eqref{Hmdof3}-\eqref{Hmdof1} will be reduced as
follows:
\begin{align}
%\frac{1}{|F|^{(n-m)/(n-j)}}
(\nabla^jv)(\delta) & \quad\forall~\delta\in\mathcal F^{n}(K), j=0,1,\cdots,m-n, \label{Hmdof0-km}\\
%h_K^j
(\frac{\partial^{|\alpha|}v}{\partial\nu_{F}^{\alpha}}, 1)_F & \quad\forall~F\in\mathcal F^{j}(K), j=1,\cdots, n-1, \alpha\in A_{j}, |\alpha|= m-j.\label{Hmdof2-km}
\end{align}
If each element $K\in \mathcal T_h$ is a simplex and $k=m=n+1$, the degrees of freedom \eqref{Hmdof0-km}-\eqref{Hmdof2-km} coincide with those of the nonconforming finite element in \cite{WuXu2019}.
\end{remark}

\subsection{Local projections}
To design the virtual element method, we first need a local $H^m$ projection.
We define a local $H^m$ projection $\Pi_k^K: H^m(K)\to\mathbb P_k(K)$ for each $K\in\mathcal T_h$ as follows: given $v\in H^m(K)$, let $\Pi_k^Kv\in\mathbb P_k(K)$ be the solution of the problem %\mnote{ define $Q_0^F$}
\begin{align}
(\nabla^m\Pi_k^Kv, \nabla^mq)_K&=(\nabla^mv, \nabla^mq)_K\quad  \forall~q\in \mathbb P_k(K),\label{eq:H2projlocal1}\\
\sum_{F\in\mathcal F^{r}(K)}Q_0^{F}(\nabla^{m-r}\Pi_k^Kv)&=\sum_{F\in\mathcal F^{r}(K)}Q_0^{F}(\nabla^{m-r}v), \quad r=1,\cdots, n-1,\label{eq:H2projlocal2} \\
\sum_{\delta\in\mathcal F^{n}(K)}(\nabla^{j}\Pi_k^Kv)(\delta)&=\sum_{\delta\in\mathcal F^{n}(K)}(\nabla^{j}v)(\delta), \quad j=0,1,\cdots, m-n.\label{eq:H2projlocal3}
\end{align}
The number of equations in~\eqref{eq:H2projlocal2}-\eqref{eq:H2projlocal3} is
\[
\sum_{r=1}^mC_{n+m-1-r}^{n-1}=C_{n+m-1}^{n}=\dim(\mathbb P_{m-1}(K)).
\]
Without causing any confusion, we will write $\Pi_k^K$ as $\Pi^K$ for simplicity.
By the similar argument as in Section~3.3 in \cite{ChenHuang2019}, we have the following results on $\Pi^K$ from the generalized Green's identity \eqref{eq:HmGreen} and Lemma~\ref{lem:bimgunisolvence}.
\begin{lemma}\label{lm:Pi}
The operator $\Pi^K: H^m(K)\to\mathbb P_k(K)$ is an $H^m$-stable projector, i.e.
\begin{equation}\label{eq:projectPk}
\Pi^Kq=q\quad\forall~q\in\mathbb P_k(K),
\end{equation}
\begin{equation}\label{eq:Hmprojbound}
|\Pi^Kv|_{m,K}\leq |v|_{m,K}\quad \forall~v\in H^m(K).
\end{equation}
And the projector $\Pi^K$ can be computed using only the degrees of freedom~\eqref{Hmdof3}-\eqref{Hmdof1}.
\end{lemma}

%Let $W_k(K) :=V_k(K)$ for $k\geq3m-1$ or $m\leq k\leq 2m - 1$.
%To compute the $L^2$ projection onto $\mathbb P_{m-1}(K)$ for $2m\leq k<3m-1$, define
%\begin{align*}
%\widetilde V_k(K)&:=\{v\in H^m(K): (-\Delta)^mv\in \mathbb P_{m-1}(K), D^{2m-j-|\alpha|}_{F, \alpha}(v)|_F\in\mathbb P_{k-(2m-j-|\alpha|)}(F),  \\
%&\qquad\qquad\qquad\quad\;\;\;\;\,\forall~ F\in\mathcal F^{j}(K),\, j=1,\cdots, n-1, \,\alpha\in A_{j} \textrm{ and } |\alpha|\leq m-j\},
%\end{align*}
%\[
%W_k(K) :=\{v\in \widetilde V_k(K): (v-\Pi^Kv, q)_K=0\quad\forall~q\in\mathbb P_{m-1,k-2m}^{\perp}(K)\}.
%\]
%We can see that $\mathbb P_k(K)\subset W_k(K)$, and the degrees of freedom~\eqref{Hmdof3}-\eqref{Hmdof1} are unisolvent for the local virtual element space $W_k(K)$ (cf. \cite[Lemma~3.8]{ChenHuang2019}).
%For $2m\leq k<3m-1$, once all the degrees of freedom~\eqref{Hmdof3}-\eqref{Hmdof1} are known, the $L^2$ projection $Q_{m-1}^K$ of any function in $W_k(K)$ is computable,  that is (cf. \cite{ChenHuang2019})
%\[
%Q_{m-1}^Kv=Q_{k-2m}^Kv+Q_{m-1}^K\Pi^{K}v-Q_{k-2m}^K\Pi^{K}v\quad\forall~v\in W_k(K).
%\]

Denote by $I_K: H^m(K)\to V_k(K)$ the canonical interpolation operator based on the degrees of freedom in~\eqref{Hmdof3}-\eqref{Hmdof1}. Due to the last statement in Lemma \ref{lm:Pi}, we have
\begin{equation}\label{eq:20181012-1}
\Pi^Kv=\Pi^K(I_Kv)\quad\forall~v\in H^m(K).
\end{equation}

\subsection{Serendipity virtual element}

Following the ideas in \cite{BeiraodaVeigaBrezziMariniRusso2016c,Russo2016,BeiraoDaVeigaBrezziDassiMariniEtAl2018}, we will give a short discussion on the reduction of the virtual element $(K, \mathcal N_k(K), V_k(K))$ by the serendipity approach in this subsection.

For ease of presentation, all the degrees of freedom~\eqref{Hmdof3}-\eqref{Hmdof1} are denoted by $\chi_{1}, \chi_{2}, \cdots, \chi_{N_K}$ in order, where $N_K$ is the dimension of $V_k(K)$.
%Let a positive integer $N_s$ be less than or equal to $N_K$.
%Take the first $N_s$ of the degrees of freedom~\eqref{Hmdof3}-\eqref{Hmdof1}: $\chi_{1}, \chi_{2}, \cdots, \chi_{N_s}$.
\XH{Assume there exist some positive integer $N_s\leq N_K$, nonnegative integer $k_s\leq k$ and permutation $\sigma$ of $(1,2,\cdots, N_K)$ such that
\begin{enumerate}[(S)]
\item for any $q\in\mathbb P_{k_s}(K)$ satisfying $\chi_{\sigma(1)}(q)=\chi_{\sigma(2)}(q)=\cdots=\chi_{\sigma(N_s)}(q)=0$, it holds $q=0$.
\end{enumerate}
%\begin{enumerate}[(S1)]
%\item the degrees of freedom $\chi_{1}, \chi_{2}, \cdots, \chi_{N_s}$ contain all the ones ~\eqref{Hmdof3}-\eqref{Hmdof2}, and the moments \eqref{Hmdof1} up to a certain degree $k_s\leq k-2m$;
%\item for any $q\in\mathbb P_k(K)$ satisfying $\chi_{1}(q)=\chi_{2}(q)=\cdots=\chi_{N_s}(q)=0$, it holds $q=0$.
%\end{enumerate}
%Let
%\begin{align*}
%V_{k,k}(K)&:=\{v\in H^m(K): (-\Delta)^mv\in \mathbb P_{k}(K), D^{2m-j-|\alpha|}_{F, \alpha}(v)|_F\in\mathbb P_{k-(2m-j-|\alpha|)}(F),  \\
%&\qquad\qquad\qquad\quad\;\,\forall~ F\in\mathcal F^{j}(K),\, j=1,\cdots, n-1, \,\alpha\in A_{j} \textrm{ and } |\alpha|\leq m-j\}.
%\end{align*}
Define an operator $\Pi_k^s: V_{k}(K)\to\mathbb P_{k_s}(K)$ for each $K\in\mathcal T_h$ as
\[
\sum_{i=1}^{N_s}\chi_{\sigma(i)}(\Pi_k^sv)\chi_{\sigma(i)}(q)=\sum_{i=1}^{N_s}\chi_{\sigma(i)}(v)\chi_{\sigma(i)}(q)\quad  \forall~v\in V_k(K),\; q\in \mathbb P_{k_s}(K).
\]
The assumption (S) ensures the well-posedness of the operator $\Pi_k^s$, and
\begin{equation}\label{eq:20190726-1}
\Pi_k^sq=q\quad\forall~q\in\mathbb P_{k_s}(K).
\end{equation}
Define the space of the serendipity shape functions
\[
V_k^s(K) :=\{v\in V_{k}(K): \chi_{\sigma(i)}(v)=\chi_{\sigma(i)}(\Pi_k^sv)\quad \textrm{ for } i=N_s+1,\cdots, N_K\}.
\]
Due to \eqref{eq:20190726-1}, it holds $\mathbb P_{k_s}(K)\subseteq V_k^s(K)$.
Let $\mathcal N_k^s(K):=\{\chi_{\sigma(1)}, \chi_{\sigma(2)}, \cdots, \chi_{\sigma(N_s)}\}$, then we obtain the serendipity virtual element $(K, \mathcal N_k^s(K), V_k^s(K))$.
The well-posedness of the serendipity virtual element $(K, \mathcal N_k^s(K), V_k^s(K))$ follows from \eqref{eq:20190726-1} and the well-posedness of the virtual element $(K, \mathcal N_k(K), V_k(K))$.
%The $L^2$ projection $Q_{k}^K$ of any function in $V_k^s(K)$ is computable based on the degrees of freedom~$\mathcal N_k^s(K)$,  that is
%\[
%Q_{k}^Kv=Q_{k_s}^Kv+Q_{k}^K\Pi_k^sv-Q_{k_s}^K\Pi_k^sv\quad\forall~v\in V_k^s(K).
%\]
}

\XH{Now we give an example to illustrate the previous process. Let $n=2$, $m=3$, $k=5$ and $K$ be a triangle,
then the local degrees of freedom \eqref{Hmdof3}-\eqref{Hmdof1} will be reduced to
\begin{align*}
v(\delta), \;\nabla v(\delta) & \quad\forall~\delta\in\mathcal F^{2}(K), \\%\label{Hmdof0-kmn2}\\
(\frac{\partial^{j}v}{\partial\nu_{e,1}^{j}}, q)_e & \quad\forall~q\in\mathbb M_{j}(e), e\in\mathcal F^{1}(K), j=0,1,2.%\label{Hmdof1-kmn2}
\end{align*}
Take $k_s=4$, the reduced virtual element space $V_5^s(K)=\mathbb P_{4}(K)$ and the following reduced freedoms of freedom $\mathcal N_5^s(K)$:
\begin{align*}
v(\delta), \;\nabla v(\delta) & \quad\forall~\delta\in\mathcal F^{2}(K), \\%\label{Hmdof0-kmn2}\\
(v, 1)_e,\;(\frac{\partial^{2}v}{\partial\nu_{e,1}^{2}}, 1)_e & \quad\forall~e\in\mathcal F^{1}(K).%\label{Hmdof1-kmn2}
\end{align*}
The assumption (S) for $\mathcal N_5^s(K)$ and $\mathbb P_{4}(K)$ holds due to Lemma~3.1 in \cite{HuZhang2017}. Indeed the serendipity virtual element $(K, \mathcal N_5^s(K), V_5^s(K))$ is exactly the $H^3$-nonconforming finite element of the first case (2.2) in \cite{HuZhang2017}.
}

\XH{Therefore, with a suitable choice of the degrees of freedom $\mathcal N_k^s(K)$, the serendipity virtual element $(K, \mathcal N_k^s(K), V_k^s(K))$ may reduce to an $H^m$-nonconforming finite element, i.e. $V_k^s(K)=\mathbb P_{k_s}(K)$. We point out that it is not easy to verify the  assumption~(S). However, it gives a hint to
recover some existing finite elements and construct new $H^m$-nonconforming finite elements.
}

\section{Discrete Method}

We will present the $H^m$-nonconforming virtual element method for the polyharmonic equation based on the virtual element $(K, \mathcal N_k(K), V_k(K))$
in this section.

\subsection{Discretization}
Consider the polyharmonic equation with homogeneous Dirichlet boundary condition
\begin{equation}\label{eq:polyharmonic}
\begin{cases}
(-\Delta)^mu=f\qquad\qquad\qquad\quad\text{in
}\Omega, \\
u=\frac{\partial u}{\partial \nu}=\cdots=\frac{\partial^{m-1} u}{\partial \nu^{m-1}}=0 \quad\text{on
}\partial\Omega,
\end{cases}
\end{equation}
%\begin{equation}
%\left\{
%\begin{aligned}
%&(-\Delta)^mu=f\quad\quad\;\text{in
%}\Omega, \\
%&u=\partial_{n}u=0\;\;\,\text{on }\partial\Omega,
%\end{aligned}
%\right.  \label{eq:polyharmonic}%
%\end{equation}
where $f\in L^2(\Omega)$ and $\Omega\subset \mathbb{R}^n$ with $m>n\geq2$. % and $\partial_{n}u$ is the normal derivative of $u$.
The variational formulation of the polyharmonic equation~\eqref{eq:polyharmonic} is to find $u\in H_0^m(\Omega)$ such that
\[
(\nabla^mu, \nabla^mv)=(f, v)\quad \forall~v\in H_0^m(\Omega).
\]

%For ease of presentation, we let $W_k^C(K):=W_k^L(K)$, $\mathcal N_k^C(K):=\mathcal N_k^L(K)$ and $\Pi^{C,K}:=\Pi^{L,K}$ for $m=1$.
Let the global $H^m$-nonconforming virtual element space be
\begin{align*}
V_h:=\{v_h\in L^2(\Omega): & v_h|_K\in V_k(K)\textrm{ for each } K\in\mathcal T_h; \; \textrm{all the degrees of } \\
&\textrm{freedom \eqref{Hmdof3}-\eqref{Hmdof2} are continuous across each } F\in\mathcal{F}_h^{r, i}, \\
&\qquad\qquad\;\;\textrm{ and vanish on each }  F\in\mathcal{F}_h^{r, \partial} \textrm{ for } r=1,\cdots, n\}.
\end{align*}

To introduce the bilinear form, let the stabilization
\begin{align*}
S_K(w, v):=&\,h_K^{n-2m} \sum_{j=1}^{n-1}\sum_{F\in\mathcal F^j(K)}\sum_{\alpha\in A_{j}\atop|\alpha|\leq m-j}\sum_{i=1}^{N_{F,k-(2m-j-|\alpha|)}}\chi_{j,i}^{F,\alpha}(w)\chi_{j,i}^{F,\alpha}(v) \\
&+ h_K^{n-2m}\sum_{\delta\in\mathcal F^{n}(K)}\sum_{j=0}^{m-n}h_K^{2j}(\nabla^jw)(\delta):(\nabla^jv)(\delta).
\end{align*}
\XH{The stabilization term $S_K(\cdot, \cdot)$ only includes the boundary degrees of freedom,
whereas all the degrees of freedom are involved in the stabilization term in \cite{ChenHuang2019}.}
Define the local bilinear form $a_{h,K}(\cdot, \cdot): V_k(K)\times V_k(K)\to\mathbb R$ as
\[
a_{h,K}(w, v):=(\nabla^m\Pi^{K}w, \nabla^m\Pi^{K}v)_K+S_{K}(w-\Pi^{K}w, v-\Pi^{K}v),
\]
and the global bilinear form $a_h(\cdot, \cdot): V_h\times V_h\to\mathbb R$ as
\[
a_h(w_h, v_h):=\sum_{K\in\mathcal T_h}a_{h,K}(w_h, v_h).
\]

To present the right hand side, for any nonnegative integer $\ell$, denote by $Q_h^l$ the $L^2$-orthogonal projection onto $\mathbb P_{l}(\mathcal T_h)$. Define $\Pi_h: H^{m}(\mathcal T_h) \to \mathbb P_{k}(\mathcal T_h)$ as follows: given $v\in H^{m}(\mathcal T_h)$,
\[
(\Pi_h v) |_{K} : = \Pi^K (v|_{K})\quad\forall~K\in\mathcal T_h.
\]
Then the right hand side is given by
$$
\langle f, v_h \rangle : =
\begin{cases}
\, (f, \Pi_h v_h ), & m\leq k \leq 2m -1, \\
\, (f, Q_h^{m-1}\Pi_h v_h+Q_h^{k-2m}(v_h-\Pi_h v_h)), & 2m \leq k \leq 3m - 2,\\
\, (f, Q_h^{k-2m} v_h ), & 3m-1\leq k.
\end{cases}
$$
%\[
%Q_{m-1}^Kv=Q_{k-2m}^Kv+Q_{m-1}^K\Pi^{K}v-Q_{k-2m}^K\Pi^{K}v\quad\forall~v\in W_k(K).
%\]
Notice that when $2m \leq k \leq 3m - 2$, it holds
\begin{equation}\label{eq:20190805-1}
(Q_h^{m-1}\Pi_h v_h+Q_h^{k-2m}(v_h-\Pi_h v_h), q)=(v_h, q)\quad\forall~q\in \mathbb P_{k-2m}(\mathcal T_h).
\end{equation}

Combining previous components leads to
the $H^m$-nonconforming virtual element method for the polyharmonic equation~\eqref{eq:polyharmonic} in any dimension: find $u_h\in V_h$ such that
\begin{equation}\label{polyharmonicNoncfmVEM}
a_h(u_h, v_h)=\langle f, v_h\rangle\quad \forall~v_h\in V_h.
\end{equation}

\begin{remark}\rm
The virtual element method \eqref{polyharmonicNoncfmVEM} is completely determined by the degrees of freedom $\mathcal N_k(K)$, i.e.~\eqref{Hmdof3}-\eqref{Hmdof1}.
The space of shape functions $V_k(K)$ is not necessary in the definition and thus the implementation of the virtual element method \eqref{polyharmonicNoncfmVEM}.
Introducing $V_k(K)$ is merely for the purpose of analysis, thus the space of shape functions $V_k(K)$ is virtual.
\end{remark}

\begin{remark}\rm
Differently from \cite{ChenHuang2019}, the global virtual element space $V_h$ in this paper is defined directly from $V_K$ rather than some modification of $V_K$, since indeed we do not need the computable $L^2$-projection $Q_h^{m-1}v_h$ for any $v_h\in V_h$.
\end{remark}

\subsection{Inverse inequality and Poincar\'e inequality}
For any $K\in\mathcal T_h$, let $B_K$ be the maximal ball with respect to which $K$ is star-shaped, and $K_s\subset\mathbb R^n$ be the regular inscribed simplex of $B_K$, where all the edges of $K_s$ share the common length.

\begin{lemma}\label{lem:polynomialequiv}
Assume the mesh $\mathcal T_h$ satisfies condition (A1) and $K\in\mathcal T_h$.
It holds for any nonnegative integer $\ell$ that
\begin{equation}\label{eqn:polynomialequiv}
\|q\|_{0, K}\eqsim \|q\|_{0, K_s}\quad\forall~q\in \mathbb P_{\ell}(K).
\end{equation}
\end{lemma}
\begin{proof}
Taking any $\bs x\in K$, let
\[
v(t):=q(\bs x_{K_s}+t(\bs x-\bs x_{K_s}))\quad\forall~t\in[0,1].
\]
Then $v(t)$ is a polynomial of degree $\ell$ on the interval $[0,1]$. By the Taylor's theorem, it follows
\[
q(\bs x)=v(1)=\sum_{i=0}^{\ell}\frac{v^{(i)}(0)}{i!}.
\]
Since $v^{(i)}(0)=\nabla^iq(\bs x_{K_s})(\bs x-\bs x_{K_s},\cdots, \bs x-\bs x_{K_s})$, we get from the inverse inequality of polynomials on $K_s$
\[
|v^{(i)}(0)|\lesssim h_K^i\|\nabla^iq\|_{L^{\infty}(K_s)}\lesssim \|q\|_{L^{\infty}(K_s)}\lesssim h_{K_s}^{-n/2}\|q\|_{0,K_s}.
\]
Thus we have
\[
|q(\bs x)|\lesssim h_{K_s}^{-n/2}\|q\|_{0,K_s}\quad\forall~\bs x\in K,
\]
which implies
\[
\|q\|_{0,K}\lesssim h_{K}^{n/2}\|q\|_{L^{\infty}(K)}\lesssim \|q\|_{0,K_s}.
\]
The another side of \eqref{eqn:polynomialequiv} is clear.
\end{proof}

\begin{lemma}
Assume the mesh $\mathcal T_h$ satisfies condition (A1) and $K\in\mathcal T_h$.
It holds for any nonnegative integers $\ell$ and $i$ that
\begin{equation}\label{eq:polyinverse}
 \| q \|_{0,K} \lesssim h_K^{-i}\| q \|_{-i,K} \quad \forall~q\in \mathbb P_{\ell}(K).
\end{equation}
\end{lemma}
\begin{proof}
Applying \eqref{eqn:polynomialequiv} and the inverse inequality of polynomials on simplices, it follows
\[
\|q\|_{0,K}\lesssim \|q\|_{0, K_s}\lesssim h_K^{-i}\| q \|_{-i,K_s},
\]
which yields \eqref{eq:polyinverse}.
\end{proof}

Applying the trace inequality \eqref{L2trace} and the same argument used in the proof of Lemma~A.5 in \cite{ChenHuang2019}, we get
 the Poincar\'e inequality on the kernel of the local $H^m$ projection $\Pi^K$ under the mesh conditions (A1)-(A2)
\begin{equation}\label{eq:poincare}
\sum_{j=0}^{n}\sum_{s=0}^{m-j}\sum_{F\in\mathcal F^j(K)}h_K^{s+j/2}\|\nabla^s v\|_{0,F}\lesssim h_K^m\|\nabla^m v\|_{0,K}\quad\forall~v\in \ker(\Pi^K),\; K\in\mathcal T_h,
\end{equation}
where $\ker(\Pi^K):=\{v\in H^m(K): \Pi^Kv=0\}$.

\subsection{Norm equivalence}

\begin{lemma}
Assume the mesh $\mathcal T_h$ satisfies conditions (A1) and (A2).
For any $K\in\mathcal T_h$, we have
\begin{equation}\label{eq:20190122-1}
S_K(v,v)\lesssim \sum_{j=0}^mh_K^{2(j-m)}|v|_{j,K}^2\quad\forall~v\in V_k(K),
\end{equation}
\begin{equation}\label{eq:SKequiv2}
S_K(v,v)\lesssim \|\nabla^m v\|_{0,K}^2\quad\forall~v\in V_k(K)\cap\ker(\Pi^K).
\end{equation}
\end{lemma}
\begin{proof}
We get from the proof of Lemma~A.6 in \cite{ChenHuang2019}
%Similarly as \eqref{eq:20190120-3}, we get
\begin{align*}
&h_K^{n-2m} \sum_{j=1}^{n-1}\sum_{F\in\mathcal F^j(K)}\sum_{\alpha\in A_{j}\atop|\alpha|\leq m-j}\sum_{i=1}^{N_{F,k-(2m-j-|\alpha|)}}\left(\chi_{j,i}^{F,\alpha}\right)^2(v) \\
\lesssim & \sum_{j=1}^{n-1}\sum_{F\in\mathcal F^j(K)}\sum_{\alpha\in A_{j}\atop|\alpha|\leq m-j}\sum_{i=1}^{N_{F,k-(2m-j-|\alpha|)}}h_K^{2|\alpha|-2m+j} \|\nabla^{|\alpha|}v\|_{0,F}^2.
\end{align*}
Due to the trace inequality \eqref{L2trace}, it follows
\begin{align*}
\sum_{\delta\in\mathcal F^{n}(K)}\sum_{i=0}^{m-n}h_K^{n-2m+2i}|\nabla^iv(\delta)|^2 & \lesssim \sum_{e\in\mathcal F^{n-1}(K)}\sum_{i=0}^{m-n}h_K^{n-2m+2i-1}\|\nabla^i v\|_{0,e}^2 \\
&\quad\; + \sum_{e\in\mathcal F^{n-1}(K)}\sum_{i=0}^{m-n}h_K^{n-2m+2i+1}\|\nabla^{i+1} v\|_{0,e}^2.
\end{align*}
%Due to the Sobolev inequality \eqref{eq:Sobolevinequality}, it follows
%\[
%\sum_{\delta\in\mathcal F^{n}(K)}\sum_{i=0}^{m-n}h_K^{n-2m+2i}|\nabla^iv(\delta)|^2\lesssim \sum_{F\in\mathcal F^{n-2}(K)}\sum_{i=0}^{m-n+2}h_K^{n-2m+2i-2} \|\nabla^{i}v\|_{0,F}^2.
%\]
Then we have
\[
S_K(v,v)\lesssim \sum_{j=1}^{n-1}\sum_{F\in\mathcal F^j(K)}\sum_{i=0}^{m-j}h_K^{2i-2m+j} \|\nabla^{i}v\|_{0,F}^2.
\]
Hence we acquire \eqref{eq:20190122-1} by applying the trace inequality \eqref{L2trace} recursively.
Finally we conclude \eqref{eq:SKequiv2} from \eqref{eq:20190122-1} and the Poincar\'e inequality \eqref{eq:poincare}.
\end{proof}

We then consider another side of the norm equivalence.
Take an element $K\in\mathcal T_h$.
%For any $F\in\mathcal{F}^j(K)$ with $j\geq 1$, let $\mathbb R_F^{n-j}$ be the $(n-j)$-dimensional affine space passing through $F$, $\mathcal{F}_F^j(K):=\{e\in\mathcal{F}^j(K): e\subset\mathbb R_F^{n-j}\}$, and
%$$
%\lambda_{F,i}:=\nu_{F,i}^{\intercal}\frac{\boldsymbol x-\boldsymbol x_F}{h_K},\quad i=1,\cdots, j.
%$$
%Apparently $\lambda_{F,i}|_F=0$, i.e. the $(n-1)$-dimensional face $\lambda_{F,i}=0$ passes through $F$.
%If $K$ is a simplex and $F\in \mathcal F^{1}(K)$, $\lambda_{F,1}$ is just the barycenter coordinate when $h_F$ represents the height of $K$ corresponding to the base $F$.
%And $\mathcal{F}_F^j(K)=\{F\}$ if $K$ is strictly convex.
%For any $F\in\mathcal{F}^j(K)$ with $j\geq 1$, and $F'\in\mathcal{F}^j(K)\backslash\mathcal{F}_F^j(K)$, let $\nu_{F,F'}$ be some unit normal vector of $F'$ such that the hyperplane $\nu_{F,F'}^{\intercal}(\bs x-\bs x_{F'})=0$ does not pass through $F$.
Define a bubble function
%$$
%b_\delta:=\prod_{\delta'\in\mathcal{F}^n(K)\backslash\{\delta\}}\frac{(\boldsymbol x_{\delta}-\boldsymbol x_{\delta'})^{\intercal}(\boldsymbol x-\boldsymbol x_{\delta'})}{h_K|\boldsymbol x_{\delta}-\boldsymbol x_{\delta'}|},
%$$
$$
b_\delta(\boldsymbol x):=\prod_{\delta'\in\mathcal{F}^n(K)\backslash\{\delta\}}\frac{(\boldsymbol x_{\delta}-\boldsymbol x_{\delta'})^{\intercal}(\boldsymbol x-\boldsymbol x_{\delta'})}{|\boldsymbol x_{\delta}-\boldsymbol x_{\delta'}|^2},
$$
for each $\delta\in \mathcal F^{n}(K)$. Apparently we have $b_\delta(\delta)=1$, and $b_\delta(\delta')=0$ for each $\delta'\in\mathcal{F}^n(K)\backslash\{\delta\}$.
%For any $F\in\mathcal{F}_h^j$ with $1\leq j<n$
%and any nonnegative integer $\ell$, let $N_{F,\ell}:=\#\mathbb M_{\ell}(F)$, and all the functions in $\mathbb M_{\ell}(F)$ be $\{\varphi_{F,i}\}_{i=1}^{N_{F,\ell}}$.

\begin{lemma}
Assume the mesh $\mathcal T_h$ satisfies condition (A1). Take any $F\in\mathcal{F}_h^j$ with $1\leq j<n$.
%For any nonnegative integer $k$, let $\mathbb M_k(F)=\{\varphi_i\}_{i=1}^{N_{F,k}}$ with $N_{F,k}:=\#\mathbb M_k(F)$.
The following norm equivalence holds
\begin{equation}\label{eq:polynorm}
\|q\|_{0,F}^2\eqsim  h_F^{n-j}\sum_{i=1}^{N_{F,k}} q_{i}^2\qquad\forall~ q := \sum_{i=1}^{N_{F,k}} q_{i}\XH{\textsf{m}_{F, i}}\in \mathbb P_k(F),
\end{equation}
where $q_{i}\in\mathbb R$.
\end{lemma}
\begin{proof}
Noticing that $F_s$ is a simplex, it follows from \eqref{eqn:polynomialequiv}, the scaling argument and the norm equivalence of the finite-dimensional space
\[
\|q\|_{0,F}^2\eqsim \|q\|_{0,F_s}^2\eqsim  h_{F_s}^{n-j}\sum_{i=1}^{N_{F,k}} q_{i}^2,
\]
which gives \eqref{eq:polynorm}.
\end{proof}

\begin{lemma}
Assume the mesh $\mathcal T_h$ satisfies conditions (A1) and (A2).
Let $K\in\mathcal T_h$.
We have for any $v\in V_k(K)$ that
\begin{align}
h_K^m\|(-\Delta)^mv\|_{0,K} + \sum_{\delta\in\mathcal{F}^n(K)}\sum_{\alpha\in A_{n}\atop|\alpha|\leq m-n}h_K^{m-|\alpha|-n/2}|D^{2m-n-|\alpha|}_{\delta,\alpha}(v)|& \notag\\
 + \sum_{j=1}^{n-1}\sum_{F\in\mathcal{F}^j(K)}\sum_{\alpha\in A_{j}\atop|\alpha|\leq m-j}h_K^{m-|\alpha|-j/2}\|D^{2m-j-|\alpha|}_{F, \alpha}(v)\|_{0,F}&\lesssim \|\nabla^mv\|_{0,K}. \label{eq:20190221-4}
\end{align}
\end{lemma}
\begin{proof}
Adopting the same argument as in the proof of Lemmas~A.1-A.2 in \cite{ChenHuang2019}, we get
\begin{align}
\sum_{j=1}^{n-1}\sum_{F\in\mathcal{F}^j(K)}\sum_{\alpha\in A_{j}\atop|\alpha|\leq m-j}h_K^{m-|\alpha|-j/2}\|D^{2m-j-|\alpha|}_{F, \alpha}(v)\|_{0,F}& \notag\\
 +h_K^m\|(-\Delta)^mv\|_{0,K} &\lesssim \|\nabla^mv\|_{0,K}. \label{eq:20190221-5}
\end{align}
Now consider $\delta\in\mathcal{F}^n(K)$ and $\alpha\in A_{n}$ with $|\alpha|\leq m-n$.
Notice that $D^{2m-n-|\alpha|}_{\delta,\alpha}(v)$ is a constant, which can be regarded as the constant function in $\mathbb R^n$.
Let
\[
\phi_{\delta}(\bs x):=\frac{1}{\alpha!}b_{\delta}^{2m}D^{2m-n-|\alpha|}_{\delta,\alpha}(v)\prod\limits_{i=1}^n(\nu_{\delta,i}^{\intercal}(\boldsymbol x-\boldsymbol x_{\delta}))^{\alpha_i},
\]
where $\alpha!=\alpha_1!\cdots\alpha_n!$,
then we have
\begin{equation}\label{eq:20181025-2}
\|\phi_{\delta}\|_{0,K}\lesssim h_{K}^{|\alpha|+n/2}|D^{2m-n-|\alpha|}_{\delta,\alpha}(v)|,
\end{equation}
\[
\frac{\partial^{|\alpha|}\phi_{\delta}}{\partial\nu_{\delta}^{\alpha}}(\delta)=(b_{\delta}(\delta))^{2m}D^{2m-n-|\alpha|}_{\delta,\alpha}(v)=D^{2m-n-|\alpha|}_{\delta,\alpha}(v).
\]
Hence
\begin{equation}\label{eq:20181025-1}
\left|D^{2m-n-|\alpha|}_{\delta,\alpha}(v)\right|^2= D^{2m-n-|\alpha|}_{\delta,\alpha}(v)\frac{\partial^{|\alpha|}\phi_{\delta}}{\partial\nu_{\delta}^{\alpha}}(\delta).
\end{equation}
Noticing that $b_{\delta}(\delta')=0$ for all $\delta'\in\mathcal{F}^n(K)\backslash\{\delta\}$, we have for each $\delta'\in\mathcal{F}^n(K)\backslash\{\delta\}$
$$
\frac{\partial^{|\beta|}\phi_{\delta}}{\partial\nu_{\delta'}^{\beta}}(\delta')=0\quad\forall~\beta\in A_{n} \;\textrm{ with }\; |\beta|\leq m-n.
$$
For any $\beta\in A_{n}$, $|\beta|<|\alpha|$, since $\displaystyle\frac{\partial^{|\beta|}}{\partial\nu_{\delta}^{\beta}}\bigg(\prod\limits_{i=1}^n(\nu_{\delta,i}^{\intercal}(\boldsymbol x-\boldsymbol x_{\delta}))^{\alpha_i}\bigg)(\delta)=0$, it yields
$\displaystyle
\frac{\partial^{|\beta|}\phi_{\delta}}{\partial\nu_{\delta}^{\beta}}(\delta)=0.
$
For any $\beta\in A_{n}$, $|\beta|=|\alpha|$, but $\beta\neq\alpha$, noting that $\displaystyle\frac{\partial(\nu_{\delta,i}^{\intercal}(\boldsymbol x-\boldsymbol x_{\delta}))}{\partial\nu_{\delta,\ell}}=0$ for $i\neq\ell$, we also have
$\displaystyle
\frac{\partial^{|\beta|}\phi_{\delta}}{\partial\nu_{\delta}^{\beta}}(\delta)=0.
$
Based on the previous discussion, we obtain from \eqref{eq:20181025-1}, the generalized Green's identity \eqref{eq:HmGreen} and the density argument
\begin{align*}
\left|D^{2m-n-|\alpha|}_{\delta,\alpha}(v)\right|^2&=(\nabla^mv, \nabla^m\phi_{\delta})_K-((-\Delta)^mv, \phi_{\delta})_K \\
&\quad- \sum_{j=1}^{n-1}\sum_{F\in\mathcal F^j(K)}\sum_{\beta\in A_{j}\atop|\beta|\leq m-j}\Big (D^{2m-j-|\beta|}_{F, \beta}(v), \frac{\partial^{|\beta|}\phi_{\delta}}{\partial\nu_{F}^{\beta}}\Big )_F\\
&\quad- \sum_{\beta\in A_{n}\atop|\alpha|<|\beta|\leq m-n}D^{2m-n-|\beta|}_{\delta, \beta}(v) \frac{\partial^{|\beta|}\phi_{\delta}}{\partial\nu_{\delta}^{\beta}}(\delta).
\end{align*}
Employing the Cauchy-Schwarz inequality, the inverse inequality \eqref{eq:polyinverse} and \eqref{eqn:polynomialequiv}, it follows
\begin{align*}
&\left|D^{2m-n-|\alpha|}_{\delta,\alpha}(v)\right|^2\\
\lesssim &h_K^{-m}\|\nabla^mv\|_{0,K}\|\phi_{\delta}\|_{0,K} + \|(-\Delta)^mv\|_{0,K}\|\phi_{\delta}\|_{0,K}\\
&+ \sum_{j=1}^{n-1}\sum_{F\in\mathcal F^j(K)}\sum_{\beta\in A_{j}\atop|\beta|\leq m-j}h_K^{-|\beta|-j/2}\Big \|D^{2m-j-|\beta|}_{F, \beta}(v)\Big\|_{0,F}\|\phi_{\delta}\|_{0,K} \\
&+ \sum_{\beta\in A_{n}\atop|\alpha|<|\beta|\leq m-n}h_K^{-|\beta|-n/2}\left|D^{2m-n-|\beta|}_{\delta,\beta}(v)\right|\|\phi_{\delta}\|_{0,K},
\end{align*}
which combined with \eqref{eq:20181025-2} yields
\begin{align*}
h_K^{m-|\alpha|-n/2}\left|D^{2m-n-|\alpha|}_{\delta,\alpha}(v)\right|
\lesssim &\|\nabla^mv\|_{0,K} + h_K^{m}\|(-\Delta)^mv\|_{0,K}\\
&+ \sum_{j=1}^{n-1}\sum_{F\in\mathcal F^j(K)}\sum_{\beta\in A_{j}\atop|\beta|\leq m-j}h_K^{m-|\beta|-j/2}\Big \|D^{2m-j-|\beta|}_{F, \beta}(v)\Big\|_{0,F} \\
&+ \sum_{\beta\in A_{n}\atop|\alpha|<|\beta|\leq m-n}h_K^{m-|\beta|-n/2}\left|D^{2m-n-|\beta|}_{\delta,\beta}(v)\right|.
\end{align*}
Then we get from \eqref{eq:20190221-5}
\[
h_K^{m-|\alpha|-n/2}\left|D^{2m-n-|\alpha|}_{\delta,\alpha}(v)\right|
\lesssim \|\nabla^mv\|_{0,K} + \sum_{\beta\in A_{n}\atop|\alpha|<|\beta|\leq m-n}h_K^{m-|\beta|-n/2}\left|D^{2m-n-|\beta|}_{\delta,\beta}(v)\right|.
\]
Finally applying this inequality recursively gives
\[
h_K^{m-|\alpha|-n/2}\left|D^{2m-n-|\alpha|}_{\delta,\alpha}(v)\right|
\lesssim \|\nabla^mv\|_{0,K},
\]
which together with \eqref{eq:20190221-5} implies \eqref{eq:20190221-4}.
\end{proof}

\begin{lemma}\label{lemma:20181113-1}
Assume the mesh $\mathcal T_h$ satisfies condition (A1) and $K\in\mathcal T_h$.
For any $q\in \mathbb P_{k-2m}(K)$, there exists $p\in\mathbb P_{k}(K)$ satisfying
\[
(-\Delta)^mp=q, \quad\textrm{ and }\quad \|\nabla^{i}p\|_{0,K}\lesssim h_K^{2m-i}\|q\|_{0,K}
\]
for any nonnegative integer $i$.
\end{lemma}
\begin{proof}
Since $\Delta: \mathbb P_{\ell+2}(K_s)\to\mathbb P_{\ell}(K_s)$ is surjective for any nonnegative integer $\ell$, the operator $(-\Delta)^m: \mathbb P_{k}(K_s)\to\mathbb P_{k-2m}(K_s)$ is surjective. Thus the operator $(-\Delta)^m: \mathbb P_{k}(K_s)/\ker((-\Delta)^m)\to\mathbb P_{k-2m}(K_s)$ is an isomorphism. Then there exists $p\in\mathbb P_{k}(K_s)$ such that
\begin{equation}\label{eqn:20181113-1}
(-\Delta)^mp=q|_{K_s},
\end{equation}
and by the scaling argument,
\[
\|p\|_{0,K_s}\lesssim h_K^{2m}\|q\|_{0,K_s}.
\]
Notice that $p\in\mathbb P_{k}(K_s)$ is spontaneously regarded as a polynomial in $\mathbb P_{k}(K)$.
Applying the inverse inequality~\eqref{eq:polyinverse} and \eqref{eqn:polynomialequiv}, we have
\[
\|\nabla^{i}p\|_{0,K}\lesssim h_K^{-i}\|p\|_{0,K}\lesssim h_K^{-i}\|p\|_{0,K_s}\lesssim h_K^{2m-i}\|q\|_{0,K_s}\leq h_K^{2m-i}\|q\|_{0,K}.
\]
The identity \eqref{eqn:20181113-1} implies $\left((-\Delta)^mp-q\right)|_{K_s}=0$, which together with the fact $(-\Delta)^mp-q\in\mathbb P_{k-2m}(K)$ ends the proof.
\end{proof}

\begin{lemma}\label{lemma:SKequivVK}
Assume the mesh $\mathcal T_h$ satisfies conditions (A1) and (A2).
For any $K\in\mathcal T_h$, it holds
\begin{equation}\label{eq:SKequivVK}
\|\nabla^m v\|_{0,K}^2\lesssim S_K(v,v)\quad\forall~v\in V_k(K)\cap\ker(\Pi^K).
\end{equation}
\end{lemma}
\begin{proof}
Employing Lemma~\ref{lemma:20181113-1}, there exists $p\in\mathbb P_{k}(K)$ satisfying
\begin{equation}\label{eqn:20181113-2}
(-\Delta)^mp=(-\Delta)^mv,
\end{equation}
\begin{equation}\label{eq:20190222-1}
\|\nabla^{m}p\|_{0,K}\lesssim h_K^{m}\|(-\Delta)^mv\|_{0,K}\lesssim \|\nabla^mv\|_{0,K},
\end{equation}
in which we have used \eqref{eq:20190221-4}.
By the definition of $\Pi^K$, it follows
\[
\|\nabla^m v\|_{0,K}^2 =(\nabla^mv, \nabla^mv)_K =(\nabla^m(v-p), \nabla^mv)_K.
\]
And exploiting the generalized Green's identity \eqref{eq:HmGreen} and \eqref{eqn:20181113-2}, we have
\begin{align}
\|\nabla^m v\|_{0,K}^2 = &\sum_{j=1}^{n-1}\sum_{F\in\mathcal F^j(K)}\sum_{\alpha\in A_{j}\atop|\alpha|\leq m-j}\Big (D^{2m-j-|\alpha|}_{F, \alpha}(v-p), \frac{\partial^{|\alpha|}v}{\partial\nu_{F}^{\alpha}}\Big )_F \notag\\
&+ \sum_{\delta\in\mathcal F^{n}(K)}\sum_{\alpha\in A_{n}\atop|\alpha|\leq m-n}D^{2m-n-|\alpha|}_{\delta,\alpha}(v-p)\frac{\partial^{|\alpha|}v}{\partial\nu_{\delta}^{\alpha}}(\delta).\label{eq:20181025-4}
\end{align}
Since $v\in V_k(K)$, we have $D^{2m-j-|\alpha|}_{F, \alpha}(v-p)|_F\in\mathbb P_{k-(2m-j-|\alpha|)}(F)$ for any $F\in\mathcal F^{j}(K)$.
Then there exist constants $c_i$, $i=1,\cdots, N_{F, k-(2m-j-|\alpha|)}$ such that
$$
\Big (D^{2m-j-|\alpha|}_{F, \alpha}(v-p), \frac{\partial^{|\alpha|}v}{\partial\nu_{F}^{\alpha}}\Big )_F=h_K^{n-j-|\alpha|}\sum_{i=1}^{N_{F, k-(2m-j-|\alpha|)}}c_i\chi_{j,i}^{F,\alpha}(v)
$$
Applying the norm equivalence on the polynomial space $\mathbb P_{k-(2m-j-|\alpha|)}(F)$, cf. \eqref{eq:polynorm}, we get
$$
\|D^{2m-j-|\alpha|}_{F, \alpha}(v-p)\|_{0,F}^2\eqsim h_K^{n-j}\sum_{i=1}^{N_{F, k-(2m-j-|\alpha|)}}c_i^2.
$$
Hence it follows from \eqref{eq:20190221-4} and \eqref{eq:20190222-1}
\begin{align*}
&\Big (D^{2m-j-|\alpha|}_{F, \alpha}(v-p), \frac{\partial^{|\alpha|}v}{\partial\nu_{F}^{\alpha}}\Big )_F\\
\lesssim &h_K^{(n-j)/2-|\alpha|}\|D^{2m-j-|\alpha|}_{F, \alpha}(v-p)\|_{0,F}\sqrt{\sum_{i=1}^{N_{F,k-(2m-j-|\alpha|)}}\left(\chi_{j,i}^{F,\alpha}\right)^2(v)} \\
\lesssim&  \|\nabla^m(v-p)\|_{0,K}\sqrt{S_K(v,v)}\lesssim\|\nabla^mv\|_{0,K}\sqrt{S_K(v,v)}.
\end{align*}
Applying \eqref{eq:20190221-4} and \eqref{eq:20190222-1} again, it holds for each $\delta\in\mathcal F^{n}(K)$, and $\alpha\in A_{n}$ with $|\alpha|\leq m-n$
\[
D^{2m-n-|\alpha|}_{\delta,\alpha}(v-p)\frac{\partial^{|\alpha|}v}{\partial\nu_{\delta}^{\alpha}}(\delta)\lesssim \|\nabla^m(v-p)\|_{0,K}\sqrt{S_K(v,v)}\lesssim \|\nabla^mv\|_{0,K}\sqrt{S_K(v,v)}.
\]
Therefore we conclude \eqref{eq:SKequivVK} from \eqref{eq:20181025-4} and the last two inequalities.
\end{proof}

\XH{
\begin{remark}\rm
The reason of the stabilization term $S_K(\cdot,\cdot)$ only involving the boundary degrees of freedom is that
the operator $(-\Delta)^m: \mathbb P_{k}(K)\to\mathbb P_{k-2m}(K)$ is onto and has a continuous right inverse (cf. Lemma~\ref{lemma:20181113-1}).
\end{remark}}

%\begin{lemma}
%Assume the mesh $\mathcal T_h$ satisfies conditions (A1) and (A2).
%For any $K\in\mathcal T_h$, it holds
%\begin{equation}\label{eq:SKequiv1}
%\|\nabla^m v\|_{0,K}^2\lesssim S_K(v,v)\quad\forall~v\in W_k(K)\cap\ker(\Pi^K).
%\end{equation}
%\end{lemma}
%\begin{proof}
%Due to \eqref{eq:SKequivVK} and the definition of $W_k(K)$, we only need to prove \eqref{eq:SKequiv1} for $2m\leq k<3m-1$.
%Applying Lemma~\ref{lemma:20181113-1} again, there exists $p\in\mathbb P_{k}(K)$ satisfying
%\begin{equation}\label{eqn:20181113-3}
%(-\Delta)^mp=Q_{k-2m}^K((-\Delta)^mv), \;\textrm{ and }\; \|\nabla^{m}p\|_{0,K}\lesssim h_K^{m}\|(-\Delta)^mv\|_{0,K}.
%\end{equation}
%Since it holds from the definition of $W_k(K)$
%\[
%((-\Delta)^m(v-p),v)=((-\Delta)^mv-Q_{k-2m}^K((-\Delta)^mv),v)=0,
%\]
%we get from the definition of $\Pi^K$ and the generalized Green's identity \eqref{eq:HmGreen}
%\begin{align*}
%\|\nabla^m v\|_{0,K}^2= &(\nabla^mv, \nabla^mv)_K =(\nabla^m(v-p), \nabla^mv)_K\\
%= & \sum_{j=1}^{n-1}\sum_{F\in\mathcal F^j(K)}\sum_{\alpha\in A_{j}\atop|\alpha|\leq m-j}\Big (D^{2m-j-|\alpha|}_{F, \alpha}(v-p), \frac{\partial^{|\alpha|}v}{\partial\nu_{F}^{\alpha}}\Big )_F \notag\\
%&+ \sum_{\delta\in\mathcal F^{n}(K)}\sum_{\alpha\in A_{n}\atop|\alpha|\leq m-n}D^{2m-n-|\alpha|}_{\delta,\alpha}(v-p)\frac{\partial^{|\alpha|}v}{\partial\nu_{\delta}^{\alpha}}(\delta).
%\end{align*}
%Thus we can finish the proof using the same argument as in proof of Lemma~\ref{lemma:SKequivVK}.
%\end{proof}

At last, combining \eqref{eq:SKequiv2} and \eqref{eq:SKequivVK} gives the norm equivalence \eqref{eq:SKequiv}.
\begin{theorem}
Assume the mesh $\mathcal T_h$ satisfies conditions (A1) and (A2). For any $K\in\mathcal T_h$, the following norm equivalence holds
\begin{equation}\label{eq:SKequiv}
S_K(v,v)\eqsim \|\nabla^m v\|_{0,K}^2\quad\forall~v\in V_k(K)\cap\ker(\Pi^K),
\end{equation}
where the constant is independent of $h_K$, but may depend on the chunkiness parameter $\rho_K$, the degree of polynomials $k$, the order of differentiation $m$, the dimension of space $n$, and the shape regularity and quasi-uniform constants of the virtual triangulation $\mathcal T^*_h$.
\end{theorem}

From now on, we always assume the mesh $\mathcal T_h$ satisfies conditions (A1) and (A2).
By the Cauchy-Schwarz inequality and the norm equivalence~\eqref{eq:SKequiv}, we have
\begin{equation}\label{eq:SKbound}
S_K(w, v)\lesssim  |w|_{m, K}|v|_{m, K}\quad\forall~w, v\in V_k(K)\cap\ker(\Pi^K).
\end{equation}
%By~\eqref{eq:Hmprojbound} and~\eqref{eq:SKbound}, we have
%\begin{equation*}%\label{eq:ahKbound}
%a_{h,K}(w, v)\lesssim  |w|_{m, K}|v|_{m, K}\quad\forall~w, v\in W_k(K),
%\end{equation*}
which implies the continuity of $a_h(\cdot,\cdot)$
\begin{equation}\label{eq:ahbound}
a_h(w_h, v_h)\lesssim  |w_h|_{m, h}|v_h|_{m, h}\quad\forall~w_h, v_h\in V_h+\mathbb P_k(\mathcal T_h).
\end{equation}
Next we verify the coercivity of $a_h(\cdot,\cdot)$.
\begin{lemma}
For any $v_h\in V_h+\mathbb P_k(\mathcal T_h)$, it holds
\begin{equation}\label{eq:ahcoercivity}
|v_h|_{m,h}^2\lesssim a_h(v_h, v_h).
\end{equation}
\end{lemma}
\begin{proof}
Since $\Pi^K$ is the $H^m$-orthogonal projection,
\[
|v_h|_{m, K}^2 = \left|\Pi^K(v_h|_K)\right|_{m, K}^2 + \left|v_h-\Pi^K(v_h|_K)\right|_{m, K}^2.
\]
Applying~\eqref{eq:SKequiv}, we have
\begin{align}
|v_h|_{m, K}^2 &\lesssim  \left|\Pi^K(v_h|_K)\right|_{m, K}^2 + S_K(v_h-\Pi^K(v_h|_K), v_h-\Pi^K(v_h|_K))\notag\\
&=a_{h,K}(v_h, v_h), \label{eq:20181012-3}
\end{align}
which implies~\eqref{eq:ahcoercivity}.
\end{proof}

Therefore the nonconforming virtual element method~\eqref{polyharmonicNoncfmVEM} is uniquely solvable by the Lax-Milgram lemma.

\subsection{Weak continuity}
Based on Lemma~\ref{lem:bimgunisolvence}, for any $F\in\mathcal F_h^{1}$, $v_h\in V_h$, we have the weak continuity
\begin{equation}\label{eq:weakcontinuitygradient}
(\llbracket\nabla_h^sv_h\rrbracket, \tau)_F =0 \quad\forall~\tau\in \mathbb P_{k-(2m-1-s)}(F; \mathbb T_{n}(s))
\end{equation}
for $s=0,1,\cdots, m-1$,
and
%\begin{equation}\label{eq:weakcontinuitygradient2}
%Q_0^e(\llbracket\nabla_h^{m-r}v_h\rrbracket|_F) =0 \quad\forall~e\in\mathcal F^{r-1}(F)
%\end{equation}
%for $r=1,2,\cdots, n-1$,
\begin{equation}\label{eq:weakcontinuitygradient2}
Q_0^e(\llbracket\nabla_h^{s}v_h\rrbracket|_F) =0 \quad\forall~e\in\mathcal F^{m-s-1}(F)
\end{equation}
for $s=m-n,\cdots, m-1$,
where $\nabla_h$ is the elementwise gradient with respect to the partition $\mathcal T_h$.
%We shall derive some bound on the jump $\llbracket\nabla_h^sv_h\rrbracket$ using the weak continuity and the trace inequality.
%The required bound can be derived directly from the weak continuity \eqref{eq:weakcontinuitygradient} when $s\geq\frac{3m-k}{2}-1$.

%By the weak continuity \eqref{eq:weakcontinuitygradient},
%the mean value of $\nabla_h^sv_h$ over $F$ is continuous only when $s\geq 2m-1-k$.
%For $s<2m-1-k$, the mean value of $\nabla_h^sv_h$ is merely continuous over some low-dimensional face of $F$, cf. \eqref{eq:weakcontinuitygradient2}.
%As a concrete example, consider the Morley element in three dimensions. The mean value of $\nabla_hv_h$ over faces is continuous, but the mean value of $v_h$ is only continuous on edges rather than over faces.
%some moment (maybe over faces of $F$) of $\llbracket\nabla_h^sv_h\rrbracket$ is zero across each $F\in \mathcal{F}_h^{1}$.

Recall the following error estimates of the $L^2$ projection and the Bramble-Hilbert Lemma (cf.~\cite[Lemma~4.3.8]{BrennerScott2008}).
\begin{lemma}
Let $\ell\geq 0$. For each $F\in\mathcal F_h^{r}$ with $r=0,1,\cdots, n-1$, and $e\in\mathcal{F}^1(F)$, we have for any $v\in H^{\ell+1}(F)$ that %$F\in\mathcal F_h^{1}$ with $F\subset\partial K$, we have
\begin{align}
\label{eq:PKerror}
\|v-Q_{\ell}^Fv\|_{0,F}&\lesssim h_F^{\ell+1}|v|_{\ell+1, F},\\
\label{eq:PFerror}
\|v-Q_{\ell}^ev\|_{0,e}&\lesssim h_F^{\ell+1/2}|v|_{\ell+1, F}.
\end{align}
For each $K\in\mathcal T_h\cup \mathcal T_h^*$, there exists a linear operator $T_{\ell}^K: L^1(K)\to \mathbb P_{\ell}(K)$ such that for any $v\in H^{\ell+1}(K)$
\begin{equation}\label{eq:Bramble-Hilbert}
\|v-T_{\ell}^Kv\|_{j,K}\lesssim h_K^{\ell+1-j}|v|_{\ell+1, K}\quad\textrm{ for  }\; 0\leq j\leq \ell+1.
\end{equation}
\end{lemma}

%Notice that the constants in \eqref{eq:PKerror}-\eqref{eq:Bramble-Hilbert} depend on the star-shaped constant, i.e.
%the supremum of
%the chunkiness parameter $\rho_K$. %$\mathrm{diam}(K)/ \rho_K$, where $\rho_K$ is the supremum of all the radius of the balls with respect to which the element $K$ is star-shaped.

Define $T_h: L^2(\Omega)\to \mathbb P_k(\mathcal T_h)$ as
\[
(T_hv)|_K:= T_k^{K}(v|_K)\quad\forall~K\in\mathcal T_h.
\]

\begin{lemma}\label{lem:20181013-1}
Given $F\in\mathcal F_h^{1}$ and nonnegative integer $s<m-n$,  it holds for any $v_h\in V_h$
\begin{equation}\label{eq:20190701-3}
\sum_{j=0}^{n-2}\sum_{e_j\in\mathcal F^j(F)}h_F^{j/2}\big\|\llbracket\nabla_h^{s}v_h\rrbracket\big\|_{0,e_{j}}\lesssim
\sum_{j=0}^{n-2}\sum_{e_j\in\mathcal F^j(F)}h_F^{j/2+m-n-s}\big\|\llbracket\nabla_h^{m-n}v_h\rrbracket\big\|_{0,e_{j}}.
\end{equation}
\end{lemma}
\begin{proof}
For $j=1, \cdots, n-2$, applying the trace inequality \eqref{L2trace}, it follows
\begin{align*}
&\sum_{e_j\in\mathcal F^j(F)}h_F^{j/2}\big\|\llbracket\nabla_h^{s}v_h\rrbracket-Q_0^{F}(\llbracket\nabla_h^{s}v_h\rrbracket)\big\|_{0,e_{j}} \\
\lesssim & \sum_{e_{j-1}\in\mathcal F^{j-1}(F)}h_F^{(j-1)/2}\big\|\llbracket\nabla_h^{s}v_h\rrbracket-Q_0^{F}(\llbracket\nabla_h^{s}v_h\rrbracket)\big\|_{0,e_{j-1}} \\
&+\sum_{e_{j-1}\in\mathcal F^{j-1}(F)}h_F^{(j+1)/2}\big\|\llbracket\nabla_h^{s+1}v_h\rrbracket\big\|_{0,e_{j-1}}.
\end{align*}
Then employing this inequality recursively, we obtain
\begin{align*}
&\sum_{e_j\in\mathcal F^j(F)}h_F^{j/2}\big\|\llbracket\nabla_h^{s}v_h\rrbracket-Q_0^{F}(\llbracket\nabla_h^{s}v_h\rrbracket)\big\|_{0,e_{j}} \\
\lesssim & \big\|\llbracket\nabla_h^{s}v_h\rrbracket-Q_0^{F}(\llbracket\nabla_h^{s}v_h\rrbracket)\big\|_{0,F} +
\sum_{i=0}^{j-1}\sum_{e_i\in\mathcal F^i(F)}h_F^{(i+2)/2}\big\|\llbracket\nabla_h^{s+1}v_h\rrbracket\big\|_{0,e_{i}}.
\end{align*}
Hence we get from \eqref{eq:PKerror}
\begin{align}
&\sum_{j=0}^{n-2}\sum_{e_j\in\mathcal F^j(F)}h_F^{j/2}\big\|\llbracket\nabla_h^{s}v_h\rrbracket-Q_0^{F}(\llbracket\nabla_h^{s}v_h\rrbracket)\big\|_{0,e_{j}} \notag\\
\lesssim & \big\|\llbracket\nabla_h^{s}v_h\rrbracket-Q_0^{F}(\llbracket\nabla_h^{s}v_h\rrbracket)\big\|_{0,F} +
\sum_{i=0}^{n-3}\sum_{e_i\in\mathcal F^i(F)}h_F^{(i+2)/2}\big\|\llbracket\nabla_h^{s+1}v_h\rrbracket\big\|_{0,e_{i}} \notag\\
\lesssim &
\sum_{j=0}^{n-3}\sum_{e_j\in\mathcal F^j(F)}h_F^{(j+2)/2}\big\|\llbracket\nabla_h^{s+1}v_h\rrbracket\big\|_{0,e_{j}}.\label{eq:20190701-2}
\end{align}
Adopting the trace inequality \eqref{L2trace}, it follows from \eqref{eq:20190701-2}
\begin{align*}
&\sum_{\delta\in\mathcal F^{n-1}(F)}h_F^{(n-1)/2}\big|\llbracket\nabla_h^sv_h\rrbracket-Q_0^F(\llbracket\nabla_h^sv_h\rrbracket)\big|(\delta)\\
 \lesssim & \sum_{e_{n-2}\in\mathcal F^{n-2}(F)}h_F^{(n-2)/2}\big\|\llbracket\nabla_h^sv_h\rrbracket-Q_0^F(\llbracket\nabla_h^sv_h\rrbracket)\big\|_{0, e_{n-2}} + h_F^{n/2}\big\|\llbracket\nabla_h^{s+1}v_h\rrbracket\big\|_{0, e_{n-2}}\\
\lesssim &
\sum_{j=0}^{n-2}\sum_{e_j\in\mathcal F^j(F)}h_F^{(j+2)/2}\big\|\llbracket\nabla_h^{s+1}v_h\rrbracket\big\|_{0,e_{j}}.
\end{align*}
Take some $\delta\in\mathcal F^{n-1}(F)$. Due to the degrees of freedom \eqref{Hmdof3}, we have $\llbracket\nabla_h^sv_h\rrbracket(\delta)=0$.
Then for $j=0,1,\cdots, n-2$ and any $e_j\in\mathcal F^j(F)$, it follows from \eqref{eq:PKerror}
\begin{align*}
h_F^{j/2}\big\|\llbracket\nabla_h^sv_h\rrbracket\big\|_{0,e_j}&=h_F^{j/2}\big\|\llbracket\nabla_h^sv_h\rrbracket - \llbracket\nabla_h^sv_h\rrbracket(\delta)\big\|_{0,e_j} \\
&=h_F^{j/2}\big\|\llbracket\nabla_h^sv_h\rrbracket - Q_0^F(\llbracket\nabla_h^sv_h\rrbracket) - (\llbracket\nabla_h^sv_h\rrbracket-Q_0^F(\llbracket\nabla_h^sv_h\rrbracket))(\delta)\big\|_{0,e_j} \\
&\leq h_F^{j/2}\big\|\llbracket\nabla_h^sv_h\rrbracket - Q_0^F(\llbracket\nabla_h^sv_h\rrbracket)\big\|_{0,e_j} \\
&\quad\;+ h_F^{(n-1)/2}\big|\llbracket\nabla_h^sv_h\rrbracket-Q_0^F(\llbracket\nabla_h^sv_h\rrbracket)\big|(\delta).
\end{align*}
Combining the last two inequalities and \eqref{eq:20190701-2} yields
\[
\sum_{j=0}^{n-2}\sum_{e_j\in\mathcal F^j(F)}h_F^{j/2}\big\|\llbracket\nabla_h^{s}v_h\rrbracket\big\|_{0,e_{j}}\lesssim
\sum_{j=0}^{n-2}\sum_{e_j\in\mathcal F^j(F)}h_F^{(j+2)/2}\big\|\llbracket\nabla_h^{s+1}v_h\rrbracket\big\|_{0,e_{j}},
\]
which indicates \eqref{eq:20190701-3}.
\end{proof}

\begin{lemma}
For each $F\in\mathcal F_h^{1}$ and nonnegative integer $s<m$, it holds %
\begin{equation}\label{eq:weakcontinuityestimate}
\big\|\llbracket\nabla_h^sv_h\rrbracket\big\|_{0,F}
\lesssim \sum_{K\in\partial^{-1}F}h_{K}^{m-s-1/2}|v_h|_{m,K}\quad\forall~v_h\in V_h.
\end{equation}
\end{lemma}
\begin{proof}
According to \eqref{eq:weakcontinuitygradient2} and the proof of Lemmas~4.5-4.6 in \cite{ChenHuang2019}, we get for $s=m-1,m-2,\cdots,m-n$ and any $e\in \mathcal F^j(F)$ with $j=0,1,\cdots,m-1-s$
\begin{equation}\label{eq:20190701-4}
\big\|\llbracket\nabla_h^sv_h\rrbracket\big\|_{0,e}
\lesssim \sum_{K\in\partial^{-1}F}h_{K}^{m-s-(j+1)/2}|v_h|_{m,K}.
\end{equation}
For $s<m-n$, it follows from \eqref{eq:20190701-3} and \eqref{eq:20190701-4} that
\begin{align*}
\big\|\llbracket\nabla_h^sv_h\rrbracket\big\|_{0,F}&\lesssim
\sum_{j=0}^{n-2}\sum_{e_j\in\mathcal F^j(F)}h_F^{j/2+m-n-s}\big\|\llbracket\nabla_h^{m-n}v_h\rrbracket\big\|_{0,e_{j}} \\
&\lesssim \sum_{K\in\partial^{-1}F}h_{K}^{m-s-1/2}|v_h|_{m,K},
\end{align*}
with together with \eqref{eq:20190701-4} with $j=0$ again implies \eqref{eq:weakcontinuityestimate}.
\end{proof}

Given the virtual triangulation $\mathcal T_h^{\ast}$, for each nonnegative integer $r<m$, define the tensorial $(m-r)$-th order Lagrange element space associated with  $\mathcal T_h^{\ast}$
\[
S_h^r:=\{\tau_h\in H_0^1(\Omega; \mathbb T_{n}(r)): \tau_h|_K\in\mathbb P_{m-r}(G; \mathbb T_{n}(r))\quad\forall~K\in\mathcal T_h^{\ast}\}.
\]

According to Lemma~4.7 in \cite{ChenHuang2019}, \eqref{eq:Bramble-Hilbert} and \eqref{eq:weakcontinuityestimate},
for $r=0,1,\cdots, m-1$ and any $v_h\in V_h$, there exists $\tau_r= \tau_r (v_h)\in S_h^r$ such that
\begin{equation}\label{eq:connectionerror}
|\nabla_h^rv_h-\tau_r|_{j,h}\lesssim h^{m-r-j}|v_h|_{m,h}
\quad \textrm{ for }\; j=0, 1, \cdots, m-r.
\end{equation}
By \eqref{eq:connectionerror},
we have the discrete Poincar\'e inequality (cf. Lemma~4.8 in \cite{ChenHuang2019})
\begin{equation}\label{eq:poincareinequality}
\|v_h\|_{m,h}\lesssim |v_h|_{m,h}\quad \forall~v_h\in V_h,
\end{equation}
and thus
\[
\|v_h\|_{m,h}\eqsim |v_h|_{m,h}\quad \forall~v_h\in V_h.
\]

\section{Error Analysis}\label{sec:erroranalysis}
In this section, we will analyze the nonconforming virtual element method~\eqref{polyharmonicNoncfmVEM}.
Denote by $I_h: H_0^m(\Omega)\to V_h$ the standard canonical interpolation operator based on the degrees of freedom in~\eqref{Hmdof3}-\eqref{Hmdof1}.
Adopting the same argument as in \cite{ChenHuang2019}, we have the following error estimate for the interpolation operator $I_h$
\begin{equation}\label{eq:Iherror}
|v-I_hv|_{m,K}\lesssim h_K^{k+1-m}|v|_{k+1, K}\quad\forall~v\in H^{k+1}(\Omega), K\in\mathcal T_h.
\end{equation}
Due to~\eqref{eq:projectPk} and~\eqref{eq:H2projlocal1}, we have the following $k$-consistency
\begin{equation}\label{eq:kconsistency}
a_{h,K}(p, v)=(\nabla^mp, \nabla^mv)_K\quad\forall~p\in \mathbb P_k(K), v\in V_k^K.
\end{equation}

\subsection{Consistency error estimate}

%To estimate the consistency error of the discretization, we split it into two cases, i.e. $k\geq 2m-1$ and $m\leq k< 2m-1$.
%For the first case,
%the weak continuity \eqref{eq:weakcontinuitygradient}, that is the projection $Q_{k-(m+i)}^F(\nabla_h^{m-(i+1)}v_h)$ is continuous across $F\in\mathcal{F}_h^1$ for $i=0,\cdots,m-1$, is sufficient to derive the optimal consistency error estimate.

\begin{lemma}\label{lem:20190725-1}
Let $u\in H_0^m(\Omega)\cap H^{k+1}(\Omega)$ be the solution of the polyharmonic equation~\eqref{eq:polyharmonic}. For $i=0,1,\cdots, \min\{m-1, k-m\}$, it holds for any $v_h\in V_h$ that
\begin{equation}\label{eq:20190725-1}
\Big |(\div^i\nabla^mu, \nabla_h^{m-i}v_h)+(\div^{i+1}\nabla^mu, \nabla_h^{m-(i+1)}v_h)\Big |\lesssim h^{k+1-m}|u|_{k+1}|v_h|_{m,h}.
\end{equation}
\end{lemma}
\begin{proof}
It follows from the weak continuity \eqref{eq:weakcontinuitygradient} with $s=m-(i+1)$ that the projection $Q_{k-(m+i)}^F(\llbracket\nabla_h^{m-(i+1)}v_h\rrbracket)=0$ for each $F\in\mathcal{F}_h^1$ and  $i=0,\cdots,m-1$.
Applying integration by parts, we get %from \eqref{eq:weakcontinuitygradient} with $s=m-(i+1)$
\begin{align*}
&(\div^i\nabla^mu, \nabla_h^{m-i}v_h)+(\div^{i+1}\nabla^mu, \nabla_h^{m-(i+1)}v_h) \\
=&\sum_{K\in\mathcal T_h}((\div^i\nabla^mu)\nu, \nabla_h^{m-(i+1)}v_h)_{\partial K} \\
=&\sum_{F\in\mathcal{F}_h^1}((\div^i\nabla^mu)\nu_{F,1}, \llbracket\nabla_h^{m-(i+1)}v_h\rrbracket)_{F} \\
=&\sum_{F\in\mathcal{F}_h^1}((\div^i\nabla^mu)\nu_{F,1}, \llbracket\nabla_h^{m-(i+1)}v_h\rrbracket-Q_{k-(m+i)}^F(\llbracket\nabla_h^{m-(i+1)}v_h\rrbracket))_{F} \\
=&\sum_{F\in\mathcal{F}_h^1}((\div^i\nabla^mu)\nu_{F,1}-Q_{k-(m+i)}^F((\div^i\nabla^mu)\nu_{F,1}), \llbracket\nabla_h^{m-(i+1)}v_h\rrbracket)_{F},
\end{align*}
with together with \eqref{eq:PFerror} and~\eqref{eq:weakcontinuityestimate} gives \eqref{eq:20190725-1}.
\end{proof}

\begin{lemma}
Let $u\in H_0^m(\Omega)\cap H^{2m-1}(\Omega)$ be the solution of the polyharmonic equation~\eqref{eq:polyharmonic}. Assume $m\leq k< 2m-1$. For $i=k-m+1,k-m+2,\cdots, m-2$, it holds for any $v_h\in V_h$ that
\begin{align}
&\Big |(\div^i\nabla^mu, \nabla_h^{m-i}v_h)+(\div^{i+1}\nabla^mu, \nabla_h^{m-(i+1)}v_h)\Big|  \notag\\
\lesssim &\Big (h^i|u|_{m+i} + h^{i+1}|u|_{m+i+1}\Big )|v_h|_{m,h}, \label{eq:20190725-2}
\end{align}
\begin{equation}\label{eq:20190725-3}
((-\div)^{m-1}\nabla^mu, \nabla_hv_h)-(f, v_h)\lesssim (h^{m-1}|u|_{2m-1} + h^{m}\|f\|_{0})|v_h|_{m,h}.
\end{equation}
\end{lemma}
\begin{proof}
Thanks to \eqref{eq:connectionerror}, for $i=k-m+1, k-m+2, \cdots, m-1$, there exists $\tau_{m-(i+1)}\in W_h^{m-(i+1)}$ such that
\begin{equation}\label{eq:temp20180809-1}
|\nabla_h^{m-(i+1)}v_h-\tau_{m-(i+1)}|_{j,h}\lesssim h^{i+1-j}|v_h|_{m,h}
\quad \textrm{ for }\; j=0, 1.
\end{equation}
Since $\tau_{m-(i+1)}\in H_0^1(\Omega; \mathbb T_{n}(m-(i+1)))$, we get for $i=k-m+1,\cdots, m-2$
\begin{equation*}%\label{eq:temp20180809-2}
(\div^i\nabla^mu, \nabla\tau_{m-(i+1)})+(\div^{i+1}\nabla^mu, \tau_{m-(i+1)})=0,
\end{equation*}
\begin{equation*}%\label{eq:temp20180809-3}
((-\div)^{m-1}\nabla^mu, \nabla\tau_{0})-(f, \tau_{0})=0.
\end{equation*}
Then we have
for $i=k-m+1,\cdots, m-2$
\begin{align*}
&\,(\div^i\nabla^mu, \nabla_h^{m-i}v_h)+(\div^{i+1}\nabla^mu, \nabla_h^{m-(i+1)}v_h) \\
=&\, (\div^i\nabla^mu, \nabla_h(\nabla_h^{m-(i+1)}v_h-\tau_{m-(i+1)})) \\
&\,+(\div^{i+1}\nabla^mu, \nabla_h^{m-(i+1)}v_h-\tau_{m-(i+1)}),
\end{align*}
\[
((-\div)^{m-1}\nabla^mu, \nabla_hv_h)-(f, v_h)=((-\div)^{m-1}\nabla^mu, \nabla_h(v_h-\tau_{0}))-(f, v_h-\tau_{0}).
\]
Hence we conclude \eqref{eq:20190725-2}-\eqref{eq:20190725-3} from \eqref{eq:temp20180809-1}.
\end{proof}

Next consider the perturbation of the right hand side.
\begin{lemma}
Assume $f\in H^{\ell}(\mathcal T_h)$ with $\ell=\max\{0,k+1-2m\}$, then it holds for any $v_h\in V_h$ that
\begin{equation}\label{eq:20190725-4}
(f, v_h)-\langle f, v_h \rangle\lesssim h^{k+1-m+\max\{0,2m-k-1\}}|f|_{\ell,h}|v_h|_{m,h}.%\label{eq:consistencyerror}
\end{equation}
\end{lemma}
\begin{proof}
For $m\leq k\leq 2m-1$, we get from the local Poincar\'e inequality \eqref{eq:poincare}
\[
(f, v_h)-\langle f, v_h \rangle=(f, v_h-\Pi_hv_h)\lesssim h^{m}\|f\|_{0}|v_h|_{m,h}.
\]
For $2m \leq k \leq 3m - 2$, it follows from \eqref{eq:20190805-1}, \eqref{eq:PKerror} and \eqref{eq:poincare}
\begin{align*}
(f, v_h)-\langle f, v_h \rangle&=\big(f, v_h-Q_h^{m-1}\Pi_h v_h-Q_h^{k-2m}(v_h-\Pi_h v_h)\big)\\
&=\big(f-Q_h^{k-2m}f, v_h-Q_h^{m-1}v_h +(Q_h^{m-1}-Q_h^{k-2m})(v_h-\Pi_h v_h)\big)\\
&\leq\|f-Q_h^{k-2m}f\|_0(\|v_h-Q_h^{m-1}v_h\|_0+\|v_h-\Pi_h v_h\|_0)\\
&\lesssim h^{k+1-m}|f|_{k+1-2m,h}|v_h|_{m,h}.
\end{align*}
For $k\geq 3m-1$, it holds from \eqref{eq:PKerror}
\begin{align*}
(f, v_h)-\langle f, v_h \rangle&=\big(f, v_h-Q_h^{k-2m}v_h\big)=\big(f-Q_h^{k-2m}f, v_h-Q_h^{k-2m}v_h\big)\\
&\leq\|f-Q_h^{k-2m}f\|_0\|v_h-Q_h^{m-1}v_h\|_0\\
&\lesssim h^{k+1-m}|f|_{k+1-2m,h}|v_h|_{m,h}.
\end{align*}
Combining the last three inequalities indicates \eqref{eq:20190725-4}.
\end{proof}

\begin{lemma}
Let $u\in H_0^m(\Omega)\cap H^{r}(\Omega)$ with $r=\max\{k+1, 2m-1\}$ be the solution of the polyharmonic equation~\eqref{eq:polyharmonic}.
Assume $f\in H^{\ell}(\mathcal T_h)$ with $\ell=\max\{0,k+1-2m\}$. It holds for any $v_h\in V_h$ that
\begin{align}
&(\nabla^mu, \nabla_h^mv_h)-\langle f, v_h \rangle\notag\\
\lesssim& h^{k+1-m}(\|u\|_{r}+h\|f\|_0+h^{\max\{0,2m-k-1\}}|f|_{\ell,h})|v_h|_{m,h}.\label{eq:consistencyerror}
\end{align}
\end{lemma}
\begin{proof}
Notice that
\begin{align*}
&(\nabla^mu, \nabla_h^mv_h)-(f, v_h)\notag\\
=&\sum_{i=0}^{m-2}(-1)^i\Big ((\div^i\nabla^mu, \nabla_h^{m-i}v_h)+(\div^{i+1}\nabla^mu, \nabla_h^{m-(i+1)}v_h)\Big )\\
&+((-\div)^{m-1}\nabla^mu, \nabla_hv_h)-(f, v_h).%\label{eq:temp20180808-1}
\end{align*}
Then it follows from~\eqref{eq:20190725-1}-\eqref{eq:20190725-3}
\[
(\nabla^mu, \nabla_h^mv_h)-(f, v_h)\lesssim h^{k+1-m}(\|u\|_{r}+h\|f\|_0)|v_h|_{m,h},
\]
which together with \eqref{eq:20190725-4} yields \eqref{eq:consistencyerror}.
\end{proof}

\subsection{Error estimate}
With previous preparation, we can show the optimal order convergence of the nonconforming virtual element method~\eqref{polyharmonicNoncfmVEM}.
\begin{theorem}
Let $u\in H_0^m(\Omega)\cap H^{r}(\Omega)$ with $r=\max\{k+1, 2m-1\}$ be the solution of the polyharmonic equation~\eqref{eq:polyharmonic}, and $u_h\in V_h$ be the nonconforming virtual element method~\eqref{polyharmonicNoncfmVEM}. Assume the mesh $\mathcal T_h$ satisfies conditions (A1) and (A2). Assume $f\in H^{\ell}(\mathcal T_h)$ with $\ell=\max\{0,k+1-2m\}$.
Then it holds
\begin{equation}\label{eq:energyerror}
|u-u_h|_{m,h}\lesssim h^{k+1-m}(\|u\|_{r}+h\|f\|_0+h^{\max\{0,2m-k-1\}}|f|_{\ell,h}).
\end{equation}
\end{theorem}
\begin{proof}
Let $v_h=I_hu-u_h$. It follows from \eqref{eq:kconsistency}, ~\eqref{eq:ahbound},~\eqref{eq:Iherror} and~\eqref{eq:Bramble-Hilbert}
\begin{align}
a_h(I_hu, v_h)-(\nabla^mu, \nabla_h^mv_h)&=a_h(I_hu-T_hu, v_h)+a_h(T_hu, v_h)-(\nabla^mu, \nabla_h^mv_h) \notag\\
&=a_h(I_hu-T_hu, v_h)+(\nabla_h^m(T_hu-u), \nabla_h^mv_h) \notag\\
&\lesssim  |I_hu-T_hu|_{m,h}|v_h|_{m,h}+|u-T_hu|_{m,h}|v_h|_{m,h} \notag\\
&\lesssim  (|u-I_hu|_{m,h}+|u-T_hu|_{m,h})|v_h|_{m,h} \notag\\
&\lesssim  h^{k+1-m}|u|_{k+1}|v_h|_{m,h}. \label{eq:temp201806181}
\end{align}
Notice that
\[
a_h(I_hu, v_h)- \langle f, v_h \rangle=a_h(I_hu, v_h)-(\nabla^mu, \nabla_h^mv_h)+(\nabla^mu, \nabla_h^mv_h)- \langle f, v_h \rangle.
\]
We get from~\eqref{eq:ahcoercivity},~\eqref{polyharmonicNoncfmVEM}, \eqref{eq:temp201806181} and~\eqref{eq:consistencyerror}
\begin{align*}
|I_hu-u_h|_{m,h}^2&\lesssim a_h(I_hu-u_h, v_h)=a_h(I_hu, v_h)- \langle f, v_h \rangle \\
&\lesssim h^{k+1-m}(\|u\|_{r}+h\|f\|_0+h^{\max\{0,2m-k-1\}}|f|_{\ell,h}),
\end{align*}
which together with \eqref{eq:Iherror} implies \eqref{eq:energyerror}.
\end{proof}

\section{Implementation of the virtual element method}

\XH{In this section, we will discuss the implementation of the nonconforming virtual element method~\eqref{polyharmonicNoncfmVEM}.
The implementation of the virtual element method of second order problems can be found in \cite{BeiraodaVeigaBrezziMariniRusso2014,BeiraodaVeigaBrezziMariniRusso2016b}.}

\XH{Take any $K\in\mathcal T_h$. Let $n_k:=\dim\mathbb P_{k}(K)$, and denote all the functions in $\mathbb M_{k}(K)$ by $\textsf{m}_1, \textsf{m}_2,\cdots, \textsf{m}_{n_k}$.
Let the bases of $V_k(K)$ be $\phi_1, \phi_2, \cdots, \phi_{N_K}$, which are dual to $\chi_1, \chi_2, \cdots, \chi_{N_K}$, i.e.
\[
\chi_i(\phi_j)=\delta_{ij}\quad i,j=1,2,\cdots, N_K.
\]
Here $\delta_{ij}$ is Kronecker delta.}

\subsection{Local $H^m$ projection}

\XH{Since $\Pi^K\phi_j\in \mathbb P_{k}(K)$ for $j=1,2,\cdots, N_K$, we can write
\begin{equation}\label{eq:PiV2Pcomponent}
\Pi^K\phi_j=\sum_{i=1}^{n_k}\pi_{ij}\textsf{m}_i.
\end{equation}
Denote the matrix representation $(\pi_{ij})_{n_k\times N_K}$ of $\Pi^K$ by $\mathbf\Pi^K$, then
\begin{equation}\label{eq:PiV2P}
(\Pi^K\phi_1, \Pi^K\phi_2, \cdots, \Pi^K\phi_{N_K})=(\textsf{m}_1, \textsf{m}_2, \cdots, \textsf{m}_{n_k})\mathbf\Pi^K.
\end{equation}
Let
$
\mathbf G:=\begin{pmatrix}
\mathbf G_{11} & \mathbf G_{12} \\
\bs O & \mathbf G_{22}
\end{pmatrix}
$, where $\bs O\in\mathbb R^{(n_{k}-n_{m-1})\times n_{m-1}}$ is the zero matrix, and matrices $\mathbf G_{11}\in\mathbb R^{n_{m-1}\times n_{m-1}}$, $\mathbf G_{12}\in\mathbb R^{n_{m-1}\times (n_{k}-n_{m-1})}$ and $\mathbf G_{22}\in\mathbb R^{(n_{k}-n_{m-1})\times(n_{k}-n_{m-1})}$ are given by
\[
\mathbf G_{11}:=\begin{pmatrix}
\sum\limits_{\delta\in\mathcal F^{n}(K)}\textsf{m}_1(\delta) & \cdots & \sum\limits_{\delta\in\mathcal F^{n}(K)}\textsf{m}_{n_{m-1}}(\delta)
\\
\sum\limits_{\delta\in\mathcal F^{n}(K)}(\nabla\textsf{m}_1)(\delta) & \cdots & \sum\limits_{\delta\in\mathcal F^{n}(K)}(\nabla\textsf{m}_{n_{m-1}})(\delta)
\\
\vdots & & \vdots\\
\sum\limits_{\delta\in\mathcal F^{n}(K)}(\nabla^{m-n}\textsf{m}_1)(\delta) & \cdots & \sum\limits_{\delta\in\mathcal F^{n}(K)}(\nabla^{m-n}\textsf{m}_{n_{m-1}})(\delta)\\
\sum\limits_{F\in\mathcal F^{n-1}(K)}Q_0^{F}(\nabla^{m-n+1}\textsf{m}_1) & & \sum\limits_{F\in\mathcal F^{n-1}(K)}Q_0^{F}(\nabla^{m-n+1}\textsf{m}_{n_{m-1}}) \\
\vdots & & \vdots\\
\sum\limits_{F\in\mathcal F^{1}(K)}Q_0^{F}(\nabla^{m-1}\textsf{m}_1) & \cdots & \sum\limits_{F\in\mathcal F^{1}(K)}Q_0^{F}(\nabla^{m-1}\textsf{m}_{n_{m-1}})
\end{pmatrix},
\]
\[
\mathbf G_{12}:=\begin{pmatrix}
\sum\limits_{\delta\in\mathcal F^{n}(K)}\textsf{m}_{n_{m-1}+1}(\delta) & \cdots & \sum\limits_{\delta\in\mathcal F^{n}(K)}\textsf{m}_{n_{k}}(\delta)
\\
\sum\limits_{\delta\in\mathcal F^{n}(K)}(\nabla\textsf{m}_{n_{m-1}+1})(\delta) & \cdots & \sum\limits_{\delta\in\mathcal F^{n}(K)}(\nabla\textsf{m}_{n_{k}})(\delta)
\\
\vdots & & \vdots\\
\sum\limits_{\delta\in\mathcal F^{n}(K)}(\nabla^{m-n}\textsf{m}_{n_{m-1}+1})(\delta) & \cdots & \sum\limits_{\delta\in\mathcal F^{n}(K)}(\nabla^{m-n}\textsf{m}_{n_{k}})(\delta)\\
\sum\limits_{F\in\mathcal F^{n-1}(K)}Q_0^{F}(\nabla^{m-n+1}\textsf{m}_{n_{m-1}+1}) & & \sum\limits_{F\in\mathcal F^{n-1}(K)}Q_0^{F}(\nabla^{m-n+1}\textsf{m}_{n_{k}}) \\
\vdots & & \vdots\\
\sum\limits_{F\in\mathcal F^{1}(K)}Q_0^{F}(\nabla^{m-1}\textsf{m}_{n_{m-1}+1}) & \cdots & \sum\limits_{F\in\mathcal F^{1}(K)}Q_0^{F}(\nabla^{m-1}\textsf{m}_{n_{k}})
\end{pmatrix},
\]
\[
\mathbf G_{22}:=\begin{pmatrix}
(\nabla^m\textsf{m}_{n_{m-1}+1}, \nabla^m\textsf{m}_{n_{m-1}+1})_K & \cdots &  (\nabla^m\textsf{m}_{n_{m-1}+1}, \nabla^m\textsf{m}_{n_{k}})_K \\
\vdots &  & \vdots \\
(\nabla^m\textsf{m}_{n_{k}}, \nabla^m\textsf{m}_{n_{m-1}+1})_K & \cdots &  (\nabla^m\textsf{m}_{n_{k}}, \nabla^m\textsf{m}_{n_{k}})_K
\end{pmatrix}.
\]
Noting that
\[
\nabla^j \textsf{m}_i= 0 \quad\textrm{for } i=1,2,\cdots, n_{j-1},\; j=1,2,\cdots, m-1,
\]
there are many zero entries in the submatrix $\mathbf G_{11}$.
Let
$
\mathbf B:=(\bs b_{1}, \bs b_{2}, \cdots, \bs b_{N_K})_{n_k\times N_K}
$ with
\[
\bs b_{j}:=\begin{pmatrix}
\sum\limits_{\delta\in\mathcal F^{n}(K)}\phi_j(\delta) \\
\sum\limits_{\delta\in\mathcal F^{n}(K)}(\nabla\phi_j)(\delta) \\
\vdots \\
\sum\limits_{\delta\in\mathcal F^{n}(K)}(\nabla^{m-n}\phi_j)(\delta) \\
\sum\limits_{F\in\mathcal F^{n-1}(K)}Q_0^{F}(\nabla^{m-n+1}\phi_j) \\
\vdots \\
\sum\limits_{F\in\mathcal F^{1}(K)}Q_0^{F}(\nabla^{m-1}\phi_j) \\
(\nabla^m\textsf{m}_{n_{m-1}+1}, \nabla^m\phi_j)_K \\
\vdots \\
(\nabla^m\textsf{m}_{n_{k}}, \nabla^m\phi_j)_K
\end{pmatrix}.
\]
Then the linear system of the problem \eqref{eq:H2projlocal1}-\eqref{eq:H2projlocal3} with $v=\phi_1, \phi_2, \cdots, \phi_{N_K}$ is
\[
\mathbf G\mathbf\Pi^K=\mathbf B.
\]
Hence we can compute $\mathbf\Pi^K$ as follows
\begin{equation}\label{eq:Pimatrixcomp}
\mathbf\Pi^K=\mathbf G^{-1}\mathbf B.
\end{equation}}

\XH{Define matrix
\[
\mathbf D:=(\mathbf D_{ij})_{N_K\times n_k}=\begin{pmatrix}
\chi_1(\textsf{m}_{1}) & \chi_1(\textsf{m}_{2}) & \cdots & \chi_1(\textsf{m}_{n_k}) \\
\chi_2(\textsf{m}_{1}) & \chi_2(\textsf{m}_{2}) & \cdots & \chi_2(\textsf{m}_{n_k}) \\
\vdots & \vdots & \cdots & \vdots \\
\chi_{N_K}(\textsf{m}_{1}) & \chi_{N_K}(\textsf{m}_{2}) & \cdots & \chi_{N_K}(\textsf{m}_{n_k})
\end{pmatrix}.
\]
Then
\begin{equation}\label{eq:PVbases}
(\textsf{m}_1, \textsf{m}_2, \cdots, \textsf{m}_{n_k})=(\phi_1, \phi_2, \cdots, \phi_{N_K})\mathbf D.
\end{equation}
%By \eqref{eq:PiV2P}, we have
%\[
%\Pi^K\phi_j=\sum_{i=1}^{n_k}\pi_{ij}\textsf{m}_i=\sum_{i=1}^{n_k}\sum_{\ell=1}^{N_K}\pi_{ij}\chi_{\ell}(\textsf{m}_i)\phi_{\ell}=\sum_{\ell=1}^{N_K}\sum_{i=1}^{n_k}\phi_{\ell}\chi_{\ell}(\textsf{m}_i)\pi_{ij},
%\]
%which combined with \eqref{eq:Pimatrixcomp} indicates
It follows from \eqref{eq:PiV2P} and \eqref{eq:Pimatrixcomp} that
\begin{align*}
&(\Pi^K\phi_1, \Pi^K\phi_2, \cdots, \Pi^K\phi_{N_K})\\
=&(\phi_1, \phi_2, \cdots, \phi_{N_K})\mathbf D\mathbf\Pi^K=(\phi_1, \phi_2, \cdots, \phi_{N_K})\mathbf D\mathbf G^{-1}\mathbf B.
\end{align*}
\begin{lemma}
It holds
\[
\mathbf G=\mathbf B\mathbf D.
\]
This provides another way to compute $\mathbf G$.
\end{lemma}
\begin{proof}
Applying \eqref{eq:PVbases}, \eqref{eq:PiV2P} and \eqref{eq:Pimatrixcomp}, we get
\begin{align*}
(\textsf{m}_1, \textsf{m}_2, \cdots, \textsf{m}_{n_k})&=(\Pi^K\textsf{m}_1, \Pi^K\textsf{m}_2, \cdots, \Pi^K\textsf{m}_{n_k}) \\
&=(\Pi^K\phi_1, \Pi^K\phi_2, \cdots, \Pi^K\phi_{N_K})\mathbf D \\
&=(\textsf{m}_1, \textsf{m}_2, \cdots, \textsf{m}_{n_k})\mathbf\Pi^K\mathbf D \\
&=(\textsf{m}_1, \textsf{m}_2, \cdots, \textsf{m}_{n_k})\mathbf\Pi^K\mathbf D \\
&=(\textsf{m}_1, \textsf{m}_2, \cdots, \textsf{m}_{n_k})\mathbf G^{-1}\mathbf B\mathbf D,
\end{align*}
as required.
\end{proof}}

\subsection{Local stiffness matrix}
\XH{
Denote the local stiffness matrix by $\mathbf A_K:=((\mathbf A_K)_{ij})\in \mathbb R^{N_K\times N_K}$, where
\[
(\mathbf A_K)_{ij}:=a_{h,K}(\phi_j, \phi_i)=(\nabla^m\Pi^{K}\phi_j, \nabla^m\Pi^{K}\phi_i)_K+S_{K}(\phi_j-\Pi^{K}\phi_j, \phi_i-\Pi^{K}\phi_i).
\]
Using \eqref{eq:PiV2Pcomponent}, the consistency term
\[
(\nabla^m\Pi^{K}\phi_j, \nabla^m\Pi^{K}\phi_i)_K=(\nabla^m\Pi^{K}\phi_j, \nabla^m\Pi^{K}\phi_i)_K=\sum_{s,r=1}^{N_K}\pi_{si}(\nabla^m\textsf{m}_r, \nabla^m\textsf{m}_s)_K\pi_{rj}.
\]
Hence the matrix representation of the consistency term is
\[
\mathbf A_K^c=(\mathbf\Pi^K)^{\intercal}\begin{pmatrix}
\bs O_{n_{m-1}\times n_{m-1}} & \bs O_{n_{m-1}\times (n_{k}-n_{m-1})} \\
\bs O_{(n_{k}-n_{m-1})\times n_{m-1}} & \mathbf G_{22}
\end{pmatrix}\mathbf\Pi^K.
\]}

\XH{Next consider the stability term. Let matrix %$N_K\times N_K$
\[
\mathbf S:=h_K^{n-2m}\begin{pmatrix}
\bs I_{(N_K-n_{k-2m})\times (N_K-n_{k-2m})} & \bs O_{(N_K-n_{k-2m})\times n_{k-2m}} \\
\bs O_{n_{k-2m}\times (N_K-n_{k-2m})} & \bs O_{n_{k-2m}\times n_{k-2m}}
\end{pmatrix}.
\]
The right-bottom zero submatrix in $\mathbf S$ reflects the fact that the stabilization $S_{K}(\cdot,\cdot)$ only involves the boundary degrees of freedom.
The stability term
\[
S_{K}(\phi_j-\Pi^{K}\phi_j, \phi_i-\Pi^{K}\phi_i)=\sum_{s,r=1}^{N_K}(I-\mathbf D\mathbf\Pi^K)_{si}S_{K}(\phi_r, \phi_s)(I-\mathbf D\mathbf\Pi^K)_{rj}.
\]
Thus the matrix representation of the stability term is
\[
\mathbf A_K^s=(I-\mathbf D\mathbf\Pi^K)^{\intercal}\mathbf S(I-\mathbf D\mathbf\Pi^K).
\]}

\XH{Therefore the local stiffness matrix
\[
\mathbf A_K=\mathbf A_K^c+\mathbf A_K^s = (\mathbf\Pi^K)^{\intercal}\begin{pmatrix}
\bs O & \bs O \\
\bs O & \mathbf G_{22}
\end{pmatrix}\mathbf\Pi^K + (I-\mathbf D\mathbf\Pi^K)^{\intercal}\mathbf S(I-\mathbf D\mathbf\Pi^K).
\]}

\subsection{Right hand side term}
\XH{Finally we discuss the implementation of the right hand side term.
The vector representation of the right hand side term restricted on $K$ is $\mathbf b:=(b_1, b_2, \cdots, b_{N_K})^{\intercal}$ with
\[
b_i: =
\begin{cases}
\, (f, \Pi^{K}\phi_i)_K, & m\leq k \leq 2m -1, \\
\, (f, Q_{m-1}^{K}\Pi^{K}\phi_i+Q_{k-2m}^{K}(\phi_i-\Pi^{K}\phi_i))_K, & 2m \leq k \leq 3m - 2,\\
\, (f, Q_{k-2m}^{K}\phi_i)_K, & 3m-1\leq k.
\end{cases}
\]
Set
\begin{align*}
\mathbf F&:=((f, \textsf{m}_1)_K, (f, \textsf{m}_2)_K, \cdots, (f, \textsf{m}_{n_k})_K)^{\intercal}, \\
\widetilde{\mathbf F}&:=((f, \textsf{m}_1)_K, (f, \textsf{m}_2)_K, \cdots, (f, \textsf{m}_{n_{k-2m}})_K)^{\intercal}, \\
\overline{\mathbf F}&:=((f, \textsf{m}_1)_K, (f, \textsf{m}_2)_K, \cdots, (f, \textsf{m}_{n_{m-1}})_K)^{\intercal}.
\end{align*}
Let
\[
\mathbf M:=\begin{pmatrix}
(\textsf{m}_{1}, \textsf{m}_{1})_K & (\textsf{m}_{1}, \textsf{m}_{2})_K & \cdots &  (\textsf{m}_{1}, \textsf{m}_{n_{k}})_K \\
(\textsf{m}_{2}, \textsf{m}_{1})_K & (\textsf{m}_{2}, \textsf{m}_{2})_K & \cdots &  (\textsf{m}_{2}, \textsf{m}_{n_{k}})_K \\
\vdots & \vdots &  & \vdots \\
(\textsf{m}_{n_{m-1}}, \textsf{m}_{1})_K & (\textsf{m}_{n_{m-1}}, \textsf{m}_{2})_K & \cdots &  (\textsf{m}_{n_{m-1}}, \textsf{m}_{n_{k}})_K
\end{pmatrix},
\]
\[
\widetilde{\mathbf M}:=\begin{pmatrix}
(\textsf{m}_{1}, \textsf{m}_{1})_K & (\textsf{m}_{1}, \textsf{m}_{2})_K & \cdots &  (\textsf{m}_{1}, \textsf{m}_{n_{k-2m}})_K \\
(\textsf{m}_{2}, \textsf{m}_{1})_K & (\textsf{m}_{2}, \textsf{m}_{2})_K & \cdots &  (\textsf{m}_{2}, \textsf{m}_{n_{k-2m}})_K \\
\vdots & \vdots &  & \vdots \\
(\textsf{m}_{n_{k-2m}}, \textsf{m}_{1})_K & (\textsf{m}_{n_{k-2m}}, \textsf{m}_{2})_K & \cdots &  (\textsf{m}_{n_{k-2m}}, \textsf{m}_{n_{k-2m}})_K
\end{pmatrix},
\]
\[
\overline{\mathbf M}:=\begin{pmatrix}
(\textsf{m}_{1}, \textsf{m}_{1})_K & (\textsf{m}_{1}, \textsf{m}_{2})_K & \cdots &  (\textsf{m}_{1}, \textsf{m}_{n_{m-1}})_K \\
(\textsf{m}_{2}, \textsf{m}_{1})_K & (\textsf{m}_{2}, \textsf{m}_{2})_K & \cdots &  (\textsf{m}_{2}, \textsf{m}_{n_{m-1}})_K \\
\vdots & \vdots &  & \vdots \\
(\textsf{m}_{n_{m-1}}, \textsf{m}_{1})_K & (\textsf{m}_{n_{m-1}}, \textsf{m}_{2})_K & \cdots &  (\textsf{m}_{n_{m-1}}, \textsf{m}_{n_{m-1}})_K
\end{pmatrix}.
\]
%\[
%\widetilde{\mathbf D}:=(\widetilde{\mathbf D}_{ij})_{N_K\times n_{k-2m}}=\begin{pmatrix}
%\chi_1(\textsf{m}_{1}) & \chi_1(\textsf{m}_{2}) & \cdots & \chi_1(\textsf{m}_{n_{k-2m}}) \\
%\chi_2(\textsf{m}_{1}) & \chi_2(\textsf{m}_{2}) & \cdots & \chi_2(\textsf{m}_{n_{k-2m}}) \\
%\vdots & \vdots & \cdots & \vdots \\
%\chi_{N_K}(\textsf{m}_{1}) & \chi_{N_K}(\textsf{m}_{2}) & \cdots & \chi_{N_K}(\textsf{m}_{n_{k-2m}})
%\end{pmatrix}.
%\]
}

\XH{For $m\leq k \leq 2m -1$, it follows from \eqref{eq:PiV2Pcomponent} that
\[
b_i=(f, \Pi^{K}\phi_i)=\sum_{j=1}^{n_k}\pi_{ji}(f, \textsf{m}_j)
\]
Thus we have
\[
\mathbf b=(\mathbf\Pi^K)^{\intercal}\mathbf F.
\]}

\XH{For $k\geq 3m-1$,
since it holds for positive integer $j\leq n_{k-2m}$ that
\[
(Q_{k-2m}^{K}\phi_i, \textsf{m}_j)_K=(\phi_i, \textsf{m}_j)_K=|K|\chi_{N_K-n_{k-2m}+j}(\phi_i)=|K|\delta_{N_K-n_{k-2m}+j,i},
\]
we get
\begin{align*}
&(Q_{k-2m}^{K}\phi_1, Q_{k-2m}^{K}\phi_2,\cdots, Q_{k-2m}^{K}\phi_{N_K})\\
=&|K|(\textsf{m}_1, \textsf{m}_2,\cdots,\textsf{m}_{n_{k-2m}})\widetilde{\mathbf M}^{-1}(\bs O_{n_{k-2m}\times(N_K-n_{k-2m})}, I_{n_{k-2m}\times n_{k-2m}}).
\end{align*}
Hence it follows
\begin{align*}
\mathbf b&=\left((f, Q_{k-2m}^{K}\phi_1)_K, (f, Q_{k-2m}^{K}\phi_2)_K, \cdots, (f, Q_{k-2m}^{K}\phi_{N_K})_K\right)^{\intercal} \\
&=|K|(\bs O_{n_{k-2m}\times(N_K-n_{k-2m})}, \widetilde{\mathbf M}^{-1})^{\intercal}\widetilde{\mathbf F}=|K|\begin{pmatrix}
\bs O_{(N_K-n_{k-2m})\times 1}\\
\widetilde{\mathbf M}^{-1}\widetilde{\mathbf F}\end{pmatrix}.
\end{align*}}

\XH{Now consider the case $2m \leq k \leq 3m - 2$.
Noting that
\[
(\phi_1-\Pi^K\phi_1, \phi_2-\Pi^K\phi_2, \cdots, \phi_{N_K}-\Pi^K\phi_{N_K})=(\phi_1, \phi_2, \cdots, \phi_{N_K})(I-\mathbf D\mathbf\Pi^K),
\]
we obtain
\[
\begin{pmatrix}
(f, Q_{k-2m}^{K}(\phi_1-\Pi^{K}\phi_1)) \\
 \vdots \\
 (f, Q_{k-2m}^{K}(\phi_{N_K}-\Pi^{K}\phi_{N_K}))
\end{pmatrix}=|K|(I-\mathbf D\mathbf\Pi^K)^{\intercal}\begin{pmatrix}
\bs O_{(N_K-n_{k-2m})\times 1}\\
\widetilde{\mathbf M}^{-1}\widetilde{\mathbf F}\end{pmatrix}.
\]
On the other side, we have
\[
(Q_{m-1}^{K}\textsf{m}_1,\cdots,Q_{m-1}^{K}\textsf{m}_{n_k})=(\textsf{m}_1, \cdots, \textsf{m}_{n_{m-1}})\overline{\mathbf M}^{-1}\mathbf M,
\]
and thus
\[
((f, Q_{m-1}^{K}\textsf{m}_1),\cdots,(f, Q_{m-1}^{K}\textsf{m}_{n_k}))=\overline{\mathbf F}^{\intercal}\overline{\mathbf M}^{-1}\mathbf M.
\]
It follows from \eqref{eq:PiV2P} that
\[
\begin{pmatrix}
(f, Q_{m-1}^{K}\Pi^{K}\phi_1) \\
 \vdots \\
 (f, Q_{m-1}^{K}\Pi^{K}\phi_{N_K})
\end{pmatrix}=(\mathbf\Pi^K)^{\intercal}
\begin{pmatrix}
(f, Q_{m-1}^{K}\textsf{m}_1) \\
 \vdots \\
 (f, Q_{m-1}^{K}\textsf{m}_{n_k})
\end{pmatrix}=(\mathbf\Pi^K)^{\intercal}\mathbf M^{\intercal}\overline{\mathbf M}^{-1}\overline{\mathbf F}.
\]
Therefore in this case we achieve
\[
\mathbf b=(\mathbf\Pi^K)^{\intercal}\mathbf M^{\intercal}\overline{\mathbf M}^{-1}\overline{\mathbf F}+|K|(I-\mathbf D\mathbf\Pi^K)^{\intercal}\begin{pmatrix}
\bs O_{(N_K-n_{k-2m})\times 1}\\
\widetilde{\mathbf M}^{-1}\widetilde{\mathbf F}\end{pmatrix}.
\]
}

\section{Conclusion}

Based on a generalized Green's identity for $H^m$ inner product $m>n$,
we present the $H^m$-nonconforming virtual element method of any order $k$ on any shape of polytope in $\mathbb R^n$ with constraints $m>n$ and $k\geq m$ in a universal way
to continue the work in \cite{ChenHuang2019}. We improve the discrete method in \cite{ChenHuang2019} as follows:
\begin{enumerate}[(1)]
\item The stabilization term involves only the boundary degrees of freedom,
whereas all the degrees of freedom are involved in the stabilization term in \cite{ChenHuang2019};
\item For the case $2m \leq k \leq 3m - 2$, we define the right hand side term as $(f, Q_h^{m-1}\Pi_h v_h+Q_h^{k-2m}(v_h-\Pi_h v_h))$, rather than $(f, Q_h^{m-1}v_h)$ in \cite{ChenHuang2019}, as a result of which the modification of the space of shape functions is not required.
\end{enumerate}

%\section*{Acknowledgements}
%
%The author would like to thank Prof. Long Chen for the valuable discussion.

\bibliographystyle{abbrv}
%\bibliography{/Dropbox/Math/biblib/library,./paper}
%\bibliography{E:/work/paperLibrary/paper}
\bibliography{./paper}
\end{document}